\documentclass{amsart}
\usepackage{amsmath,amssymb,amsthm,amscd}
\newtheorem{theo}{Theorem}
\newtheorem{defi}[theo]{Definition}
\newtheorem{lemm}[theo]{Lemma}
\newtheorem{prop}[theo]{Proposition}
\newtheorem{cor}[theo]{Corollary}

\newtheorem{prob}[theo]{Problem}

\theoremstyle{definition}
\newtheorem{rema}[theo]{Remark}
\newtheorem{exam}[theo]{Example}
\begin{document}

\title[Maximal amenable MASAs arising from free products of hyperfinite factors]{Maximal amenable MASAs of the free group factor of two generators arising from the free products of hyperfinite factors}
\author{Koichi Shimada}
\email{kshimada@math.kyoto-u.ac.jp}
\address{Department of Mathematics, Kyoto University, Kita-shirakawa Oiwake-cho, Sakyo-ku, Kyoto, 606-8502, Japan}
\date{}
\begin{abstract}
In this paper, we give examples of maximal amenable subalgebras of the free group factor of two generators.
More precisely, we consider two copies of the hyperfinite factor $R_i$ of type $\mathrm{II}_1$. 
From each $R_i$, we take a Haar unitary $u_i$ which generates a Cartan subalgebra of it.
We  show that the von Neumann subalgebra generated by the self-adjoint operator  $u_1+u_1^{-1}+u_2+u_2^{-1}$ is maximal amenable in the free product.
This provides infinitely many non-unitary conjugate maximal amenable MASAs.
\end{abstract}

\maketitle

\section{Introduction}
In operator algebra theory, maximal amenable subalgebras  have fascinated  many researchers because they give a good insight into ambient (non-amenable) factors.
The history of the research dates back to the middle of 1960's.
In 1967, Kadison  conjectured that any maximal amenable von Neumann subalgebra of any factor of type $\mathrm{II}_1$ was isomorphic to the hyperfinite factor of type $\mathrm{II}_1$.
In the early 1980's,  Popa \cite{P} solved this conjecture negatively.
He showed that the generator subalgebra of any free group factor is maximal amenable (Even such a simple example had not been shown to be maximal amenable until his work, which tells us the difficulty of the problem).
In order to show that the subalgebra is maximal amenable, he used the ultraproduct technique and looked at a property, which is called the asymptotic orthogonality property.
Even now, his method has a strong influence upon the research on maximal amenability.

After Popa's work, his result has been generalized by many researchers.
Popa's example can be seen as an amenable free component of a free product.
In this direction, Boutonnet--Houdayer \cite{BH} reached a general structural theorem of maximal amenable subalgebras of (amalgamated) free products which may be of type III (See also Houdayer \cite{H}, Houdayer--Ueda \cite{HU} and Ozawa \cite{O}).
Popa's result can also be seen as a subalgebra arising from a subgroup of a group factor.
In this direction, Boutonnet--Carderi \cite{BC} gave a sufficient condition for a subalgebra of any  group factor coming from a subgroup to be maximal amenable.
It is also remarkable that their proof is quite concise, which relies on the study of non-normal states.

However, in order to understand the inner structure of the free group factors, it is also important to investigate subalgebras which come neither from free components nor subgroups.
The reason is that there are some non-trivial presentations of the free group factors (See Dykema \cite{D}, Guionnet--Shlyakhtenko \cite{GS} for example).
This means that the free group factors have certain flexibility, which makes them  interesting.
Hence it is important to investigate what kind of non-trivial automorphisms the free group factors admit; whether a given MASA is conjugate to the generator subalgebras by automorphisms or not.
For this purpose, it is helpful to consider whether a given MASA has similar properties to those of the generator subalgebas.
In that sense, the radial MASA of the free group factors is  an interesting example (For the definition, see Cameron--Fang--Ravichandran--White \cite{CFRW} for example).
It is not known whether the radial MASA is conjugate to one of the generator MASAs by automorphisms or not; although it is not contained in any free component or subgroup in an obvious way, there is no known property which distinguishes the radial MASA from the generator MASAs.
In particular, it was shown to be maximal amenable by \cite{CFRW} (See also  Wen \cite{W} for a simplified proof).
Their strategy for proving the maximal amenability was to determine the form of a sequence which asymptotically commutes with the radial MASA by using a basis constructed by Radulescu \cite{Rad}.
Although their strategy itself may be simple, it could not have been carried out without their tough computing power.

Motivated by this example, we present new examples of maximal amenable von Neumann subalgebras of a free group factor.
More precisely, we consider two copies of the hyperfinite factor $R_i$ ($i=1,2$) of type $\mathrm{II}_1$. 
By Dykema \cite{D}, the free product $R_1*R_2$ is isomorphic to the free group factor of two generators.
From each $R_i$, we take a Haar unitary $u_i$ which generates a Cartan subalgebra of it.
We  show that the von Neumann subalgebra generated by the self-adjoint operator  $u_1+u_1^{-1}+u_2+u_2^{-1}$ is maximal amenable in the free product.

Although our construction is similar to that of the radial MASA, it has a different aspect.
Unlike the radial MASA, our construction provides ``many'' examples.
Although we do not know whether they are really mutually  non-conjugate by automorphisms or not, they are neither mutually unitary conjugate nor conjugate by automorphisms arising from the free components.

In order to show that our subalgebras are maximal amenable, we would like to develop a similar strategy to that of Cameron--Fang--Ravichandran--White \cite{CFRW}, in which they show Popa's asymptotic orthogonality property for the radial MASA.
However, there are two problems. 
First, although their proof depends on combinatorics of the free groups, we cannot expect to find out such a good group-like structure in our setting.
Therefore, we first consider the special case, namely, the case when the Haar unitaries come from generators of the irrational rotation $\mathrm{C}^*$-algebras.
After that, we reduce the general case to the special case.
The idea of reducing the problem to that in a case when the factor comes from the irrational rotation $\mathrm{C}^*$-algebras is conceived by Ge \cite{G}.
He embedded a system of the general case into that of the special case.
However, just an imitation of Ge \cite{G} does not work well in our setting.
Our key idea for overcoming this difficulty is to embed ``asymptotically'' a subalgebra into another one.

Even after the problem is reduced to the special case, there arises another problem; combinatorics of the irrational rotation $\mathrm{C}^*$-algebras is complicated.
The irrational rotation $\mathrm{C}^*$-algebras can be seen as deformations of $\mathbf{Z}^2$.
Hence it is possible to use their algebraic structures.
However, since $\mathbf{Z}^2*\mathbf{Z}^2$ is ``less free'' than $\mathbf{F}_2$ is,  its combinatorics is more complicated, which requires further computations.

This paper is organized as follows.
In Section 2, we explain the main theorem of this paper.
Sections 3 to 7 are devoted to showing the main theorem.
In Section 3, we investigate the combinatorics for the special case and construct a good basis of $L^2$-space.
In Sections 4 and 5, we  do some analytical computations necessary to show the asymptotic orthogonality property.
In Section 6, we reduce the problem to the special case and show that the subalgebras have the asymptotic orthogonality property.
In order to show the maximal amenability of subalgebras, besides the asymptotic orthogonality property, we need to show the singularity of them.
In Section 7, we show the singularity of the subalgebras and conclude that they are maximal amenable.

\textbf{Acknowledgment.}
The author would like to thank Chenxu Wen for explaining the main idea of his recent work \cite{W} and bringing the author's attention to this topic.
He also appreciates Professor Cyril Houdayer's many pieces of useful advice, particularly introducing him Ge \cite{G}.

\section{Main theorem}
The main theorem of this paper is the following.
\begin{theo}
\label{2-1}
For each $i=1,2$, let $R_i$ be the hyperfinite factor of type $\mathrm{II}_1$ and $w_i \in R_i$ be a Haar unitary of $R_i$ which generates a Cartan subalgebra of $R_i$.
Then the von Neumann subalgebra $B$ generated by the self-adjoint operator $w_1+w_1^{-1}+w_2+w_2^{-1}$ is maximal amenable in the free product $R_1*R_2$ with respect to the traces.
\end{theo}
The following is an example of unitaries which satisfy the assumption of Theorem \ref{2-1}.
\begin{exam}
Let $\theta $ be an irrational number and $A_{\theta }$ be the universal $\mathrm{C}^*$-algebra generated by two unitaries $u,v$ with $uv=e^{2\pi i \theta}vu$ (an irrational rotation $\mathrm{C}^*$-algebra).
Then the $\mathrm{C}^*$-algebra   $A_\theta$ has a tracial state $\tau $ defined by $\tau (u^kv^l)=\delta _{k,0}\delta_{l,0}$ for $k,l\in \mathbf{Z}$.
Take the GNS representation of $A_\theta $ with respect to the trace $\tau$.
Then the weak closure $R$ of $A_\theta$ in the GNS representation is  isomorphic to the hyperfinite factor of type $\mathrm{II}_1$.
For any $k,l\in \mathbf{Z}\setminus \{ 0\}$ without any non-trivial common divisor, the unitary $u^kv^l$ is Haar and generates a Cartan subalgebra of $R$.
\end{exam}
\begin{proof}
It is possible to choose  $k',l'\in \mathbf{Z}\setminus \{0\}$ with $kl'-k'l=1$ because $k$ and $l$ do not have any non-trivial common divisor.
Then the map $u\mapsto u^kv^l$, $v\mapsto u^{k'}v^{l'}$ extends to an automorphism of $R$.
\end{proof}
We will see that Theorem \ref{2-1} has a possibility of producing many examples of maximal amenable MASAs.
We see that the theorem provides many subalgebras which are mutually non-unitary conjugate.
Let $u_i$, $v_i$ be generators of $R_i$ explained in the above example.
Let $w_i^{k,l}$ be a Haar unitary generating the von Neumann subalgebra $ \{ u_i^kv_i^l\} ''$.
Set $B_{k,l}:=\{ w_1^{k,l}+(w_1^{k,l})^{-1}+w_2^{k,l}+(w_2^{k,l})^{-1}\}'' \subset R_1*R_2$.
We show that $B_{k,l} \not \preceq B_{k',l'}$ for any $(k,l)\not =(k',l')$.
Let $a_n$ be unitaries of $B_{k,l}$ which converges weakly to $0$.
Then for any large $M>0$, any small $\epsilon >0$, there exists $N>0$ such that for any $n\geq N$, there exists a linear combination $b_n$ of the words $w$ with the following conditions.

\bigskip

(1) We have $\| a_n -b_n\| _2<\epsilon $.

(1) For any linear component $w$  of $b_n$, the length of $w$ is not smaller than $M$.

(2) For any linear component $w$ of $b_n$, the ratio of the number of $u_i$'s ($i=1,2$) in $w$  and the number of $v_i$'s in $w$  is $k/l$.

\bigskip

Take two words $x,y $ of $R_1*R_2$.
Then if we take $M>0$ large enough compared to the lengths of $x$ and $y$, for any linear component $w'$ of $xb_ny$, the ratio of  the number of $u_i$'s in  $w'$ and $v_i$'s in $w'$ is almost $k/l$.
Thus we have $E_{B_{k',l'}}(xb_ny)=0$, which implies that $B_{k,l}\not \preceq B_{k',l'}$.

In this way, it is possible to construct many non-unitary conjugate subalgebras.
However, our central interest is whether they are conjugate by automorphisms or not.
Thus we close this section with the following problem.
\begin{prob}
Are the maximal amenable subalgebras  constructed in Theorem \ref{2-1} mutually conjugate by automorphisms?
Are they conjugate to the generator MASA?
\end{prob}
Although any two Cartan subalgebras of the hyperfinite factor of type $\mathrm{II}_1$ are mutually conjugate by an automorphism, it is impossible to control the position of Haar unitaries  by the automorphism.
Hence it is not clear whether it is possible to conjugate two of them by automorphisms arising from free components.
\section{Radulescu type basis}
In order to show the main theorem, we first consider the case when the Haar unitaries come from the irrational rotation $\mathrm{C}^*$-algebras and later we reduce the general case to the special case.
Hence until the end of Section 4, we always assume the following condition.

\bigskip

\textbf{Condition.} For each $i=1,2$, there exist Haar  unitaries $u_i$, $v_i$ and a complex number $d$ of absolute value $1$ satisfying the following conditions.

\bigskip

(0) For any $n \not =0$, we have $d^n\not =1$.

(1) The factor $R_i$ is generated by $u_i$ and $v_i$ for each $i$.

(2) We have $u_iv_i=dv_iu_i$ for each $i$.

\bigskip

Set $M:=R_1*R_2$, $A:=u_1+u_1^{-1}+u_2+u_2^{-1}$.

The purpose of this section is to construct a good basis of $L^2$-space under this condition (Corollary \ref{cor-3}), which is motivated by Radulesucu \cite{Rad}.

\subsection{Decomposing $L^2M$ into three pieces}
Let $w=w_1\cdots w_n$ be a word consists of the letters $\{ u_1^{\pm 1}, u_2^{\pm 1}, v_1^{\pm 1}, v_2^{\pm 1}\}$. 
The number $n$ of letters contained in the word $w$ is said to be the \textit{word length} of $w$ and denoted by $|w|$.
We say that the word $w$ is \textit{reduced} if the word length does not decrease by finitely many times of the transformations $u_i^{\pm 1}u_i ^{\mp 1} \longleftrightarrow 1$, $v_i^{\pm 1}v_i^{\mp 1}\longleftrightarrow 1$ and $u_i ^sv_i^t\longleftrightarrow v_i^tu_i^s$ ($i=1,2$, $s,t=\pm 1$).
On the set of reduced words, we introduce an equivalence relation defined by $u_i^sv_i^t\longleftrightarrow v_i^tu_i^s$.
Then any non-trivial word $w $ is equivalent to a word of the following form.
\[ u_{i_1}^{k_1}v_{i_1}^{l_1}u_{i_2}^{k_2}v_{i_2}^{l_2}\cdots u_{i_n}^{k_n}v_{i_n}^{l_n},\]
where $n \in \mathbf{Z}_{\geq 1}$, $ i_1, \cdots , i_n \in \{ 1,2\}$ with $i_1\not =i_2\not =  \cdots \not = i_n$ (This notation means that $i_1\not =i_2$, $i_2\not =i_3, \cdots, i_{n-1}\not =i_n$), and $|k_t|+|l_t|\geq 1$ for $t=1,\cdots , n$.
A word of this form is said to be a \textit{completely reduced word}.
For each $l \geq 0$, let $\tilde{W}_l^0$ be the set of all completely reduced words with length $l$ consisting only of letters $u_i^{\pm 1}$ ($i=1,2$).
Let $\tilde{W}_l^1$ be the set of all completely reduced words with length $l$ such that if we write them as $u_{i_1}^{k_1}v_{i_1}^{l_1}u_{i_2}^{k_2}v_{i_2}^{l_2}\cdots u_{i_n}^{k_n}v_{i_n}^{l_n}$, exactly one of $l_t$'s is non-zero.
Let $\tilde{W}_l^2$ be the set of all completely reduced words with length $l$ such that if we write them as $u_{i_1}^{k_1}v_{i_1}^{l_1}u_{i_2}^{k_2}v_{i_2}^{l_2}\cdots u_{i_n}^{k_n}v_{i_n}^{l_n}$, at least two of $l_t$'s are non-zero.
We often regard a reduced word as a unitary element of $M$.
In the following, we identify the set $\bigcup _{l\geq 0, \ i=0,1,2} \tilde{W}_l^i$ as a  subset of $L^2M$.
Then it is an orthonormal basis of $L^2M$.
Set
\[ W_l^i:=\mathrm{span}\tilde{W}_l^i\]
for each $i=0,1,2$ and $l\geq 0$.
Set
\[ W_l:= W_l^0\oplus W_l^1\oplus W_l^2\]
for each $l\geq 0$.
Let $q_l$ be the orthogonal projection onto $W_l$.

\bigskip

For each $l\geq 0$, set
\[ \chi _l:=\sum _{w\in W_l^0}w,\]
which is a self-adjoint operator of $M$.
As mentioned in p.299 of Radulescu \cite{Rad}, the self-adjoint operators $\{ \chi _l\}_{l\geq 1}$ satisfy
\[ \chi _l\chi _1=\chi _1\chi _l=\chi _{l+1}+3\chi _{l-1}\]
for $l\geq 2$ and
\[ \chi _1\chi _1=\chi _2+4.\]
Notice that $A=\{ \chi _1\}''$.
Set 
\[ S_l^i :=\{ q_l (\chi _1 w), q_l(w\chi _1)\mid w\in W_k^i, k\leq l-1\} \]
for $i=0,1,2$,
\[ S_l:=S_l^0\oplus S_l^1\oplus S_l^2.\]
For $\gamma \in W_l^i $ and $w\in W_k^i$ with $k \leq l-1$, we have
\[ \langle \gamma , q_l(\chi _1w)\rangle =\langle \gamma , \chi _1 w\rangle =\langle \chi _1 \gamma , w\rangle .\]
Similarly, we have $\langle \gamma , q_l (w\chi _1)\rangle =\langle \gamma \chi _1 , w\rangle $.
Hence $\gamma \in W_l^i\ominus S_l^i$ means that when we expand $\chi _1 \gamma $ and $\gamma \chi _1$ by the orthonormal basis consisting of completely reduced words, no term with length less than $l$ appears.
Hence we have the following decomposition.
\begin{lemm}
\label{3-2}
We have 
\[ W_l\ominus S_l=(W_l^0 \ominus S_l^0) \oplus (W_l^1 \ominus S_l^1) \oplus (W_l^2\ominus S_l^2).\]
\end{lemm}
\begin{proof}
For $\xi \in W_l\ominus S_l$, let $\xi =\xi _0 +\xi_1 +\xi _2 \in W_l^0 \oplus W_l^1\oplus W_l^2$ be the orthogonal decomposition.
Since  any term of $\chi _1\xi _0$ is not canceled by any term of $\chi _1(\xi _1 +\xi _2)$ and since $\xi$ is orthogonal to $S_l$, $\chi _1\xi _0$ has no term with length less than $l$.
Similarly, the vector $\xi _0\chi_1 $ has no term with length less than $l$.
Hence the vector $\xi _0 $ is orthogonal to $S_l^0$.
Similarly, we have $\xi _1 \in (S_l^1)^\perp$ and $\xi _2 \in (S_l^2)^\perp$.
\end{proof}
For $r,s \in \mathbf{Z}_{\geq 0}$ and $\xi \in W_l$, set
\[ \xi _{r,s} :=q_{l+r+s}(\chi _r\xi \chi _s).\]
When $r<0$ or $s<0$, we define $\xi _{r,s}$ as $0$.
We would like to see what relations there are among these vectors.
The following lemma is one of them.
\begin{lemm}
\label{3-3}
For $\xi \in W_l$ with $l\geq 1$, we have
\[ \chi _1 \xi _{r,s}=\xi _{r+1,s} +3\xi _{r-1,s}\]
for $r\geq 1$, $s\geq 0$,
\[ \xi _{r,s}\chi _1 =\xi _{r,s+1}+3\xi _{r,s-1}\]
for $r\geq  0$, $s\geq 1$.
\end{lemm}
\begin{proof}
This is shown by the same argument as that of the proof of Lemma 1 (a) of Radulescu \cite{Rad}.
\end{proof}

\subsection{Structures of the space $\biggl( \bigoplus _{l}W_l^2\biggr) \oplus \biggl( \bigoplus _lW_l^0 \biggr)$}
In this subsection, we present some formulas on vectors of  the space $\biggl( \bigoplus _{l}W_l^2\biggr) \oplus \biggl( \bigoplus _lW_l^0 \biggr)$.
\begin{lemm}
\label{3-4}
For $\xi \in (W_l^1\oplus W_l ^2)\ominus S_l$ which is orthogonal to the set $\{ u_i^{\pm 1}v_i^{\pm (l-1)} \mid i=1,2\} \cup \{ v_i^{\pm l}\}$, we have 
\[ \chi _1 \xi _{0,s} =\xi _{1,s}\]
for $s\geq 0$,
\[ \xi _{r,0}\chi _1 =\xi _{r,1}\]
for $r\geq 0$.
\end{lemm}
\begin{proof}
Let
\[ \xi =\sum _{w\in \tilde{W}_l} \lambda _ww \]
be the decomposition along the orthonormal basis $\tilde{W}_l$.
We would like to show that $\langle \chi _1\xi _{0,s},w''\rangle =0$ for any $w''\in W_{k'}$ with $k'\leq l+s-1$.
In order to achieve this, we decompose $\xi $ into some pieces and compute the inner product of each component.
The word $w\in \tilde{W}_l$ may begin with $u_i^{\pm 1}$ or $v_j^{\pm 1}$.
We may consider these two kinds of terms separately.
In the following, we consider the case when $\lambda _w=0$ if $w$ begins with $v_j^{\pm 1}$ (the other case is obvious; $w\not =v_j^{\pm 1}$ implies that $q_l(\chi _1{w}_{0,s})=\chi _1{w}_{0,s}$).

\bigskip

\textbf{Claim 1.} For any word $\tilde{w}\in \tilde{W}_{l-1}$  and $w'' \in W_{k'}$ with $k'\leq l-1$, we have
\[ \langle \sum _{x\in \tilde{W}_{1}^0 \ \mathrm{with} \ x\tilde{w}\in \tilde{W}_l} \lambda _{x\tilde{w}} x\tilde{w} , q_l(\chi _1 w'')\rangle =0.\]

\textit{Proof of Claim 1.} Notice that for different $\tilde{w}$'s, the vectors $\chi _1x\tilde{w}$'s are mutually orthogonal.
 Thus no cancellation occurs among terms of $\chi _1 x \tilde{w}$ coming from different $\tilde{w}$'s.
On the other hand, by assumption, we have $\langle \xi, q_l(\chi _1w'')\rangle =0$.
Thus we get the conclusion of Claim 1.
\qed

\bigskip

Next, we show another necessary claim.

\bigskip

\textbf{Claim 2.} Let $\tilde{w}\in \tilde{W}_{l-1}$ and $x\in \tilde{W}_{1}^0$ with $x\tilde{w}\in \mathrm{W}_l$. 
Let $w'$ be a word of $W_s^0$ such that $x\tilde{w}$ does not cancel with $w'$, that is,  $|xww'|=|x|+|w|+|w'|$.
Then for any $j=1,2$, $p=\pm 1$, if we multiply $u_j^p$ from the left, it  cancel with $x\tilde{w}w'$ if and only if it cancel with $x\tilde{w}$.

\textit{Proof of Claim 2.} \textbf{Case 1.} When $x\tilde{w}\not =u_i^sv_i^{\pm (l-|s|)}$, then the word $u_j^p$ cannot commute with $x\tilde{w}$.
Hence for the word $u_j^p$, in order to cancel with $x\tilde{w}w'$, it should cancel with $x\tilde{w}$.
Conversely, if the word $u_j^p$ cancel with $x\tilde{w}$, the length of $u_j^px\tilde{w}w'$ should be strictly less than $1+l+s$.
Thus the word $u_j^p$ cancels with $x\tilde{w}w'$.

\textbf{Case 2.} When $x\tilde{w}=u_i^sv_i^{\pm (l-|s|)}$ ($s\not =0$), then this is directly checked by using $s\not =0$.
\qed

\bigskip

We also need to show the following two claims.

\textbf{Claim 3.} Let $\tilde{w}\in \tilde{W}_{l-1}$ be a word orthogonal to $v_j^{\pm (l-1)}$ ($j=1,2$) and $x\in \tilde{W}_{1}^0$ with $x\tilde{w}\in \mathrm{W}_l$. 
Then for any reduced word $w'$, we have $|x\tilde{w}w'|=|x|+|\tilde{w}|+|w'|$ if and only if $|\tilde {w}w'|=|\tilde{w}|+|w'|$.
(In this claim, the assumption that the vector $\xi$ is orthogonal to $u_i^{\pm 1}v_i^{\pm (l-1)}$, $v_i^{\pm l}$ is used).

\textit{Proof of Claim 3.} This is shown by the same way as that of the proof of Claim 2.
\qed

\bigskip 

\textbf{Claim 3'.} Let $x$ be $u_{3-j}^{\pm 1}$ and $\tilde{w} $ be $v_{j}^{\pm (l-1)}$.
Then the same conclusion as that of Claim 3 holds.

\textit{Proof of Claim 3'.} This is trivial.
\qed

By Claims 1, 2, 3 and 3', for any $w' \in \tilde{W}_s^0$ with $| \tilde{w}w'| =l+s-1$, we have
\[ \langle \sum _{x\in \tilde{W}_{1}^0 \ \mathrm{with} \ x\tilde{w}w'\in \tilde{W}_{l+s}} \lambda _{x\tilde{w}}x\tilde{w}w' , q_{l+s}(\chi _1 w'') \rangle =0\]
for any $w'' \in W_{k'}$ with $k' \leq l+s-1$.
When $\tilde{w}$ and $w'$ run over all possible values, the sum of the above equality is
\begin{align*}
 0&=\langle \sum _{w\in \tilde{W}_l} \biggl( \sum _{w' \in \tilde{W}_s^0, \ |ww'| =l+s}\lambda _ww w' \biggr) , q_{l+s}(\chi _1 w'')\rangle \\
   &=\langle \xi _{0,s}, q_{l+s}(\chi _1w'') \rangle \\
   &=\langle \chi _1 \xi _{0,s} , w''\rangle 
\end{align*}
for all $w'' \in W^2_{k'}$ with $k' \leq l+s-1$.
Hence $\chi _1 \xi _{0,s}$ has no term with length less than $l+s-1$.
Thus we have $\xi _{1,s}=q_{l+s+1}(\chi _1\xi _{0,s})=\chi _1\xi _{0,s}$.
The other equality is shown in the same way.
\end{proof}

\begin{lemm}
\label{3-5}
For $\xi \in W_l^0\ominus S_l^0$, we have the following statements.

\textup{(1)} When $l\geq 2$, we have
\[ \chi _1 \xi _{0,s} =\xi _{1,s}\]
for $s\geq 0$,
\[ \xi _{r,0}\chi _1 =\xi _{r,1}\]
for $r\geq 0$.

\textup{(2)} When $l=1$ and $\xi$ is $c_1(u_1+\epsilon u_1^{-1}) +c_2(u_2+\epsilon u_2^{-1})$ for some $\epsilon \in \{ \pm 1\}$ and $c_1,c_2\in \mathbf{C}$, we have
\[ \chi _1 \xi _{0,s}=\xi _{1,s}-\epsilon \xi _{0,s-1}\]
for $s \geq 0$,
\[ \xi _{r,0}\chi _1 =\xi _{r,1}-\epsilon \xi _{r-1,0}\]
for $r \geq 0$.
\end{lemm}
\begin{proof}
This is shown by the same argument as that of the proof of Lemma 1 (b)(c) of Radulescu \cite{Rad}.
\end{proof}

\subsection{Structures of the space  $\biggl( \bigoplus _lW_l^1\biggr)$}
In this subsection, we investigate the structure of the space $\bigoplus _{l}W_l^1$.
Fix $l\in \mathbf{Z}\setminus \{ 0\}$.
When $l\not =\pm 1$, set 
\[ \gamma ^{i,l}_{1,\pm }:=v_i^{l-\mathrm{sgn}(l)} u_i^{\pm 1},\]
\[ \gamma ^{i,l}_2 :=v_i^{l-\mathrm{sgn}(l)}(u_{3-i}+u_{3-i}^{-1}),\]
\[ \gamma ^{i,l}_3:=(u_{3-i}+u_{3-i}^{-1}) v_i^{l-\mathrm{sgn}(l)},\]
\[ \gamma ^i_l:=\frac{2}{d^{-(l-\mathrm{sgn}(l))}-d^{l-\mathrm{sgn}(l)}}(\gamma ^{i,l}_{1,+}-\gamma ^{i,l}_{1,-}) -\gamma ^{i,l}_3\]
and
\[\overline{\gamma } ^i_l:=\frac{2}{d^{l-\mathrm{sgn}(l)}-d^{-(l-\mathrm{sgn}(l))}}( d^{l-\mathrm{sgn}(l)}\gamma ^{i,l}_{1,+}-d^{-(l-\mathrm{sgn}(l))}\gamma ^{i,l}_{1,-}) -\gamma ^{i,l}_2.\]
For $l>0$, we also set
\[ W_{l}^{1, \alpha}:=\mathrm{span} \{ u_i^{\pm 1}v_j^{\pm (l-1)}, \ v_j^{\pm (l-1)}u_i^{\pm 1} \mid i=1,2, \ j=1,2\},\]
\[ W_l^{1, \beta}:=W_l^1\ominus (W_l^{1, \alpha }\oplus \mathbf{C}v_1^{\pm l }\oplus \mathbf{C}v_2^{\pm l}),\]
\[ W_l^{1,\alpha ,1}:=\mathrm{span}\{ \gamma ^i_{\pm l}, \ \overline{\gamma }^i_{\pm l} \mid i=1,2\} \]
when $l\not = 1$ and
\[ W_l^{1,\alpha ,2}:=\mathrm{span}\{ (u_{3-i}-u_{3-i}^{-1})v_i^{\pm (l-1)}, \ v_i^{\pm (l-1)}(u_{3-i}-u_{3-i}^{-1})\mid i=1,2\} \]
when $l \not =1$.
When $l=1$, we set $W_l^{1,\alpha ,1}=W_l^{1, \alpha , 2}=0$.
Obviously, the space $W_l^{1, \beta }$ is generated by the completely reduced words of the form $xv_i^{k'}x'\in \tilde{W}_l^1$ with $x,x'\in \bigcup_{k\geq 0} \tilde{W}_k^0, |x|+|x'|\geq 2, k' \not =0, i=1,2$.
\begin{lemm}
\label{7-4}
For any $l\in \mathbf{Z}_{>0}$, we have
\[ W_l^1\ominus S_l^1=(W_l^{1, \beta }\ominus S_l^1)\oplus W_l^{1, \alpha , 1}\oplus W_l^{1, \alpha ,2}\oplus \mathbf{C}v_1^{\pm l} \oplus \mathbf{C}v_2^{\pm l}.\]
\end{lemm}
\begin{proof}
First, note that we have $W_l^1\ominus S^1_l=(W^{1,\alpha }_l\ominus S^1_l)\oplus (W^{1, \beta }_l\ominus S^1_l)$.
Thus in order to get the conclusion, it is enough to determine the structure of the space $W^{1, \alpha }\ominus S^1_l$.

\bigskip

\textbf{Claim}
For $l\in \mathbf{Z}\setminus \{ 0, \pm 1\}$, let  $\xi \in W_{|l|}^1\ominus S_{|l|}^1$ be a linear combination of $v_1^{l-\mathrm{sgn}(l)}u_i^t, u_i^{t'}v_1^{l-\mathrm{sgn} (l)}$ ($i,i'=1,2$, $t,t'=\pm 1$).
Then $\xi $ is a linear combination of $\gamma _l^1$, $\overline{\gamma }_l^1 $, $(u_2-u_2^{-1})v_1^{l-\mathrm{sgn}(l)}$ and $v_1^{l-\mathrm{sgn}(l)}(u_2-u_2^{-1})$.

\textit{Proof of Claim.} 
For simplicity, we assume that $l>0$ (When $l<0$, Claim is shown in the same way).
It is possible to write $\xi$ as 
\[ c_{1,+}v_1^{l-1}u_1+c_{1,-}v_1^{l-1}u_1^{-1}+c_{2,+}v_1^{l-1}u_2+c_{2,-}v_1^{l-1}u_2^{-1}+c_{3,+}u_2v_1^{l-1}+c_{3,-}u_2^{-1}v_1^{l-1}.\]
Then since $\xi $ is orthogonal to $\chi _1w$ with $|w|\leq l-1$, we have
\[ d^{-(l-1)}c_{1,+}+d^{l-1}c_{1,-}+c_{3,+}+c_{3,-}=0.\]
Since $\xi $ is orthogonal to $w\chi _1$ with $|w|\leq l-1$, we have
\[ c_{1,+}+c_{1,-}+c_{2,+}+c_{2,-}=0.\]
Obviously, the vectors $\gamma _l^1$, $\overline{\gamma }_l^1$ , $(u_2-u_2^{-1})v_1^{l-1}$ and $v_1^{l-1}(u_2-u_2^{-1})$ satisfy the above two equalities.
On the other hand, the vector space defined by these two equalities is at most four dimensional.
Thus $\xi $  is a linear combination of $\gamma _l^1$, $\overline{\gamma }_l^1$, $(u_2-u_2^{-1})v_1^{l-1}$ and $v_1^{l-1}(u_2-u_2^{-1})$.
\qed

\bigskip

Thus $\xi \in W_l^1\ominus S_l^1$ which is a linear combination of $v_j^{\pm l}$, $v_{j'}^{\pm (l-1)}u_{i'}^{t'}$, $u_{i''}^{t''}v_{j''}^{\pm (l-1)}$ ($j,j',j'',i',i''=1,2$, $t',t''=\pm 1$) is in fact contained in $W_l^{1, \alpha ,1}\oplus W_l^{1, \alpha ,2} \oplus \mathbf{C}v_1^{\pm l}\oplus \mathbf{C} v_2^{\pm l}$.
Hence we have
\[ W_l^{1, \alpha }\ominus S_l^1=W_l^{1, \alpha , 1}\oplus W_l^{1, \alpha ,2} \oplus \mathbf{C}v_1^{\pm 1} \oplus \mathbf{C}v_2^{\pm 1},\]
which implies the conclusion of the lemma.
\end{proof}

\begin{lemm}
\label{8}
For any $n,m \geq 0$, $i=1,2$, we have
\[ \chi _n\gamma ^{i,l}_j\chi _m \in \mathrm{span} \{ (\gamma ^{i,l}_j)_{r,s} \mid j=(1, \pm ), 2,3, \ r,s \geq 0\}\vee \mathbf{C} v_i^{l-1}.\]
\end{lemm}
\begin{proof}
Set
\[ V_i:=\mathrm{span} \{ (\gamma _j^{i,l})_{r,s} \mid j=(1, \pm ), 2,3, \ r,s \geq 0\}.\]
Notice that we have $(v_i^{l-\mathrm{sgn}(l)})_{0,s}, (v_i^{l-\mathrm{sgn}(l)})_{r,0}\in V_i\vee \mathbf{C}v_i^{l-\mathrm{sgn}(l)}$.
For $s\geq 1$, we have 
\begin{align*}
\chi _1 (u_iv_i^{l-\mathrm{sgn}(l)})_{0,s} &= \chi _1 u_iv_i^{l-\mathrm{sgn}(l)}\biggl( \sum _{w\not =u_i^{-1}, \ |w|=s}w\biggr) \\
                                        &=(u_iv_i^{l-\mathrm{sgn}(l)})_{1,s}+v_i^{l-\mathrm{sgn}(l)}\chi _s-(v_i^{l-\mathrm{sgn}(l)}u_i^{-1})_{0,s-1},
\end{align*}
\begin{align*}
\chi _1 ((u_{3-i}+u_{3-i}^{-1})v_i^{l-\mathrm{sgn}(l)} )_{0,s} &= \chi _1(u_{3-i}+u_{3-i}^{-1})v_i^{l -\mathrm{sgn}(l)}\chi _s \\
                                                                 &=((u_{3-i}+u_{3-i}^{-1})v_i^{l-\mathrm{sgn}(l)})_{1,s} +2v_i^{l-\mathrm{sgn}(l)}\chi _s
\end{align*}
and 
\[ \chi _1(v_i^{l-\mathrm{sgn}(l)}(u_{3-i}+u_{3-i}^{-1}))_{0,s} =(v_i^{l-\mathrm{sgn}(l)}(u_{3-i}+u_{3-i}^{-1}))_{1,s}.\]
Thus we have $\chi _1 (\gamma ^{i,l}_j)_{0,s} \in V_i\vee Av_i^lA$.
Similarly, we have $(\gamma ^{i,l}_j) _{r,0}\chi _1 \in V_i\vee Av_i^lA$.
Hence by using Lemma \ref{3-3}, it is possible to conclude that $\chi _n \gamma _j^i \chi _m \in V_i \vee Av_i^lA$.
\end{proof}

\subsection{Constructing a Riesz basis of $L^2M\ominus A$}
\begin{lemm}
\label{3-6}
For $\xi \in (W_l^0\ominus S_l^0)\oplus ((W_l^{1,\alpha ,2 }\oplus W_l^{1,\beta} )\ominus S_l^1)\oplus (W_l^2\ominus S_l^2)$, we have the following statements.

\bigskip

\textup{(1)}
\textup{(i)}  For any $n,m \geq 0$ and $l\geq 2$, we have
\[ \chi _n \xi \chi _m =\xi _{n,m} -(\xi _{n,m-2}+\xi _{n-2,m})+\xi _{n-2,m-2}.\]
\textup{(ii)} For any $n,m\geq 0$ and $l\geq 2$, we have
\[ \xi _{n,m} =\sum _{r\leq n, s\leq m, \ (r,s) \ \mathrm{has} \ \mathrm{the} \ \mathrm{same} \ \mathrm{parity} \ \mathrm{as} \ \mathrm{that}\ \mathrm{of} \ (n,m)} \chi _r \xi \chi _s.\]
\textup{(2)}  When $l=1 $ and $\xi $ is of the form  $c_1(u_1+\epsilon u_1^{-1}) +c_2(u_2+\epsilon u_2^{-1})$ for some $\epsilon \in \{ \pm 1\}$ and $c_1,c_2\in \mathbf{C}$, for any $n,m \geq 0$, we have the following two statements.

\textup{(i)} We have
\begin{align*}
\chi _n \xi \chi _m=&\xi _{n,m}-(\xi _{n,m-2}+\xi _{n-2,m})+\xi _{n-2,m-2} \\
                                 &+\sum _{k\geq 2} (-\epsilon )^k (\epsilon \xi _{n-k-1,m-k+1} +\epsilon \xi _{n-k+1, m-k-1} +2\xi _{n-k,m-k}).
\end{align*}
\textup{(ii)} We have
\[ \xi _{n,m}=\sum _{r\leq n, s\leq m, \ r-s  \ \mathrm{has} \ \mathrm{the} \ \mathrm{same} \ \mathrm{parity} \ \mathrm{as} \ \mathrm{that}\ \mathrm{of} \ n-m} \epsilon ^{n-r}\chi _r \xi \chi _s.\]
\end{lemm}
\begin{proof}
By using lemmas \ref{3-3}, \ref{3-4} and \ref{3-5} , this is shown in the same way as that of Lemma 2 of Radulescu \cite{Rad}.
\end{proof}
\begin{lemm}
\label{3-7}
For $\xi \in (W_l^0\ominus S_l^0) \oplus ((W_l^{1,\alpha ,2}\oplus W_l^{1, \beta })\ominus S_l^1)\oplus (W_l^2\ominus  S_l^2) $, $l\geq 1$, the projection $p_\xi $ commutes with the projection $q_n$  and the range of $p_\xi \wedge q_n $ is the subspace of $L^2(M) $ spanned by $\{ \xi _{r,s}\mid r+s=n-l\}$. 
\end{lemm}
\begin{proof}
This is shown in the same way as the fact mentioned in the paragraph preceding to Lemma 2 of Radulescu \cite{Rad}.
However, for readers convenience, we present a proof.
Let $r_n$ be a projection onto the subspace $\mathrm{span}\{ \xi _{r,s} \mid r+s=n-l\}$.
By Lemma \ref{3-6} (1) (ii) (2) (i), we have $\xi _{r,s}\in A\xi A$.
Hence we have $r_n \leq p_\xi \wedge q_n$.
On the other hand, by Lemma \ref{3-6} (1) (i) (2) (ii), we have $p_\xi \leq \sum _nr_n$.
Hence the range of $q_np_\xi $ is contained in that of $r_n$.
Thus we have 
\[ p_\xi \wedge q_n \leq q_n p_\xi q_n \leq r_n\]
(The first equality holds without any assumption).
Hence two projections $p_\xi$ and $q_n$ commute and we have $p_\xi \wedge q_n =r_n$.
\end{proof}
\begin{lemm}
\label{3-8}
For $\xi , \xi ' \in (W_l^0\ominus S_l^0)\oplus ((W_l^{1,\alpha , 2}\oplus W_l^{1, \beta }) \ominus S_l^1) \oplus  (W_l^2\ominus S_l^2)$, we have the following statements.

\textup{(1)} When $l\geq 2$, $n,m,n',m'\geq 0$, we have
\[ \langle \xi _{n,m},  \xi ' _{n',m'} \rangle =\delta _{n,n'}\delta _{m,m'} 3^{n+m}\langle \xi , \xi '\rangle .\]

\textup{(2)} When $l=1$ and $\xi $ is of the form  $c_1(u_1+\epsilon u_1^{-1}) +c_2(u_2+\epsilon u_2^{-1})$ for some $\epsilon \in \{ \pm 1\}$ and $c_1,c_2\in \mathbf{C}$, for any $n,m,n',m'$, we have
\[ \langle \xi _{n,m},  \xi ' _{n',m'} \rangle =\delta _{\epsilon , \epsilon '} \delta _{n+m, n'+m'} 3^{n+m}(-3) ^{-|n-n'|} \langle \xi , \xi ' \rangle .\]
\end{lemm}
\begin{proof}
By using lemmas \ref{3-3}, \ref{3-4} and \ref{3-5} , this is shown in the same way as Lemma 3 of Radulescu \cite{Rad}.
\end{proof}

For $l\geq 1$, let $P_l$ be the projection onto the subspace of $L^2M$ spanned by $\{ AwA \mid w\in W_k, k\leq l-1\}$.
\begin{lemm}
\label{3-9}
\textup{(1)}  The projection $P_{l-1} $ commutes with $q_l$ and the range of $P_{l-1}q_l$ is exactly $S_l$.

\textup{(2)} For $\xi , \xi ' \in \bigcup _{l\geq 1}((W_l^0\oplus W_l^{1,\alpha ,2} \oplus W_l^{1, \beta }\oplus W_l^2) \ominus S_l)$ with $\langle \xi , \xi ' \rangle =0$, we have $A\xi A \perp A\xi ' A$.

\textup{(3)} For  $\xi  \in \bigcup _{l\geq 1}((W_l^0\oplus W_l^{1,\alpha ,2} \oplus W_l^{1, \beta }\oplus W_l^2) \ominus S_l)$, $\xi '\in W_l^{1, \alpha ,1}\oplus \mathbf{C}v_1^{\pm l} \oplus \mathbf{C}v_2^{\pm l}$, we have $A\xi A\perp A\xi 'A$.  
\end{lemm}
\begin{proof}
Although this is shown by the same argument as that of the proof of Lemma 4 of Radulescu \cite{Rad}, we present a proof .

(1) We show this by induction on $l$.
When $l=0$, then statement (1) is obvious because we have $P_{l-1}=0$.
Assume that statement (1) holds for any $k=0, \cdots , l$.
We first show the following claim.

\bigskip

\textbf{Claim.}
We have $q_{l+1}(\chi _p w\chi _q) \in S_{l+1}^i$ for any $p,q\geq 0$, $w\in W_k^i$ with $k\leq l$.

\textit{Proof of Claim}.
Let $w=w_1+w_2\in S_k^i\oplus (W_k^i\ominus S_k^i)$ be the orthogonal decomposition.
Then by the induction hypothesis, we have $w_1\in \mathrm{ran}P_{k-1}^i$.
Hence it is a finite sum of elements of the form $\chi _{p'}w'\chi _{q'}$ for some $p',q' \geq 0$, $w' \in W_{k'}^i$ ($k'<k$).
Hence by induction, it is possible to show that $w$ is a sum of  elements of the form $\chi _r\gamma \chi _s$ for some $\gamma \in W_k^i\ominus S_k^i$ with $k\leq l$, $r,s \geq 0$.
Hence we may assume that $w \in W_k^i \ominus S_k^i $ with $k\leq l$.
However, by Lemma \ref{3-6}, when $i=0,2,(1, \alpha ,2)$ and $(1, \beta)$, it is enough to show that $w _{p,q}\in S_{l+1}^i$ for any $p,q \geq 0$ with $p+q=l-k$.
This follows from Lemmas \ref{3-3}, \ref{3-4} and \ref{3-5}.
When $i=(1, \alpha ,1)$ or $w=v_i^k$, by Lemma \ref{8}, it is enough to show that  $(\gamma _j^{i,k})_{p,q} \in S_{l+1}^i$ for any $p,q\geq 0$ with $p+q=l-k$.
This follows from Lemma \ref{3-3} and the equalities of the proof of Lemma \ref{8}.
Thus we have $q_{l+1}(\chi _p\gamma \chi _q)\in S_{l+1}$.
Thus Claim holds.
\qed

\bigskip

By Claim, we have
\[ \mathrm{ran} q_{l+1}P_l ^i\subset S^i_{l+1} \subset \mathrm{ran}P_l^i\wedge q_{l+1}.\]
The second inclusion of the above is obvious.
Hence $P_l^i$ commutes with $q_{l+1}$ and the range of $P_l^i\wedge q_{l+1}$ is $S_l^i$.

(2) When $l\not =l'$, then by statement (1), we have $A\xi A\perp A\xi ' A$.
When $l=l'$, then by Lemma \ref{3-8}, we have $A\xi A\perp \xi '$.

(3) When $\xi \in W_l^{1,\alpha ,2}$, $\xi ' \in W_l^{1, \alpha ,1}$, this is shown by the direct computation.
In the other cases, this is shown by counting the number of $v_i^l$, we have $A\xi A\perp \xi '$.
\end{proof}

For each non-zero integer $l$, an integer $k$ and $r,s\geq 0$, set
\[ \xi _{r,s}^{i,l,k}:=\frac{1}{3^{\frac{r+s}{2}-1}}\sum _{w\in W_r^0, \ \mathrm{ending} \ \mathrm{with} \ u_{3-i}^{\pm 1}} wv_i^lu_1^k \sum _{w\in W_s^0, \ \mathrm{beginning } \ \mathrm{with} \ u_{3-i}^{\pm 1}}w'.\]
Set $\tilde{\chi}_1:=\chi _1/\sqrt{3}$.
Then we have the following.
\begin{lemm}
\label{5-1}
We have
\[ \tilde{\chi }_1\xi_{r,s}^{i,l,k} =\begin{cases}
    \xi _{r+1,s}^{i,l,k} +\xi _{r-1,s}^{i,l,k}  & (r \geq 1) \\
    \xi _{1,s}^{i,l,k} +\frac{1}{\sqrt{3}}(d^l\xi _{0,s}^{i,l,k+1}+d^{-l}\xi _{0,s}^{i,l,k-1}) & (r=0) .
  \end{cases}
  \]
 Similarly, we have
 \[ \xi_{r,s}^{i,l,k} \tilde{\chi }_1 =\begin{cases}
    \xi _{r,s+1}^{i,l,k} +\xi _{r,s-1}^{i,l,k}  & (s \geq 1) \\
    \xi _{r,1}^{i,l,k} +\frac{1}{\sqrt{3}}(\xi _{r,0}^{i,l,k+1}+\xi _{r,0}^{i,l,k-1}) & (s=0) .
  \end{cases}
  \]
  \end{lemm}
  \begin{proof}
  When $r\geq 1$, the first equality follows from Lemma \ref{3-3}.
  We show the first equality when $r=0$.
 We have
 \begin{align*}
 \chi _1\xi _{0,s}^{i,l,k}&=\frac{1}{3^{s/2-1}}\chi _1 v_i^l u_i^k \sum _{w\in W_s^0, \ w_1=u_{3-i}^{\pm 1}}w \\
                                 &=\sqrt{3} \xi _{1,s}^{i,l,k} +\frac{1}{3^{s/2-1}}(d^lv_i^lu_i^{k+1}+d^{-l}v_i^lu_i^{k-1})\sum _{w\in W_s^0,\ w_1=u_{3-i}^{\pm 1}}w \\
                                 &=\sqrt{3}\xi _{1,s}^{i,l,k }+d^l\xi _{0,s}^{i,l,k+1}+d^{-l}\xi _{0,s}^{i,l,k-1}.
 \end{align*}
 Thus we have
 \[ \tilde{\chi }_1\xi _{0,s}^{i,l,k} =\xi _{1,s}^{i,l,k}+\frac{1}{\sqrt{3}}(d^l\xi _{0,s}^{i,l,k+1}+d^{-l}\xi _{0,s}^{i,l,k-1}).\]
 The second equality is shown in the same way.
  \end{proof}

We have the following.
\begin{lemm}
\label{15}
\textup{(1)} For any $(i,l,k)$, we have
\[
  \| \xi _{r,s}^{i,l,k} \| _2= \begin{cases}
    2 & (r,s \geq 1) \\
    \sqrt{6} & (r=0 \ \mathrm{or} \ s=0, \ (r,s) \not =(0,0)) \\
    3 & (r=s=0).
  \end{cases}
\]

\textup{(2)} When $(i,l,k,r,s)\not =(i',l', k',r',s')$, then the vector $\xi _{r,s}^{i,l,k}$ is orthogonal to $\xi _{r',s'}^{i',l',k'}$.
\end{lemm}
\begin{proof}
(1) The number of the words with length $r $ ending with either $u_2$ or $u_2^{-1}$ is exactly $2\cdot 3 ^{r-1}$ when $r\not =0$ and $1$ when $r=0$.
Thus we have the desired equality.

(2) This is obvious.
\end{proof}

Consider a finite sum
\[ \xi :=\sum _{i=1,2, \ j=(1,\pm ), 2,3, \ r,s\geq 0, \ l\in \mathbf{Z}} a_{r,s}^{i ,l,j}(\gamma _j^{i,l})_{r,s}.\]
Let
\[ \xi =\sum _{r,s,k, i=1,2}\beta _{r,s}^{i,l,k} \xi _{r,s}^{i,l,k}\]
be the expansion along the orthonormal system $\{ \xi _{r,s}^{i,l,k}\}$ (this is always possible).
\begin{lemm}
\label{16}
Let $\xi$ and $\{ \beta _{r,s}^k\}$ be as above.
Then we have the following.

\textup{(1)} We have
\[ \beta ^{i,l,k}_{r,s}=\frac{1}{3^{|k|/2}}(\sqrt{3}\sum _{j=0}^{|k|-1}d^{l\mathrm{sgn}(k)j}\beta ^{i,l,\mathrm{sgn}(k)}_{r+j,s+|k|-j-1}+\sum _{j=1}^{|k|-1}d^{l\mathrm{sgn}(k)j}\beta ^{i,l,0}_{r+j,s+|k|-j})\]
for any $i,l,k,r,s$. 

\textup{(2)} There exists a constant $C$ such that if we have $\| \xi \|_2 \leq 1$, for any $k_0\in \mathbf{N}$, we have
\[ \frac{1}{3}\| \sum _{|k|\geq k_0}\beta _{r,s}^{i,l,k}\| _2\leq \sum _{|k|\geq k_0}|\beta _{r,s}^{i,l,k}|^2 \leq \frac{C}{3\cdot 2^{k_0}}.\]
The constant  $C$ does not depend on $\xi$ and $k_0$.
\end{lemm}
\begin{proof}
Notice that $\xi ^{i,l,k}_{r,s} \perp \xi ^{i',l',k'}_{r',s'}$, $\gamma _j^{i,l}\perp \gamma _{j'}^{i',l'}$ for any $(i,l) \not =(i',l')$.
Thus we may assume that only one $(i,l)$ appears in the sum of the definition of $\xi$.
In the lest of the proof, we denote $\alpha _{r,s}^{i,l,j}$ by $\alpha _{r,s}^{j}$, $\beta _{r,s}^{i,l,k}$ by $\beta _{r,s}^k$. 

(1) The components contributing to $\xi _{r,s}^k$ are $(\gamma ^{i,l+\mathrm{sgn}(l)}_2)_{r+|k|,s-1}$,$(\gamma ^{i,l+\mathrm{sgn}(l)}_3)_{r-1,s+|k|}$ and $(\gamma ^{i,l+\mathrm{sgn}(l)}_{1,\mathrm{sgn}(k)})_{r+j, s+|k|-j-1}$  ($j=0, \cdots , |k|-1$).
The coefficient coming from $(\gamma ^{i,l+\mathrm{sgn}(l)}_2)_{r+|k|,s-1}$ is 
\[ 3^{\frac{r+s}{2}}d^{lk}a^2_{r+|k|,s-1}.\]
The coefficient coming from $(\gamma ^{i,l+\mathrm{sgn}(l)}_3)_{r-1,s+|k|}$ is
\[ 3^{\frac{r+s}{2}}a^3_{r-1,s+|k|}.\]
The coefficient coming from $(\gamma ^{i,l+\mathrm{sgn}(l)}_{1,\mathrm{sgn}(k)})_{r+j,s+|k|-j-1}$ ($j=0, \cdots , |k|-1$) is 
\[ 3^{\frac{r+s}{2}}d^{l\mathrm{sgn}(k)j}a^{1,\mathrm{sgn}(k)}_{r+j,s+|k|-j-1}.\]
Hence the coefficient of $\xi _{r,s}^k$ is
\[ \beta _{r,s}^k=3^{\frac{r+s}{2}}\Biggl( d^{lk}a_{r+|k|,s-1}^2+a_{r-1,s+|k|}^3+\sum _{j=0}^{|k|-1}d^{l\mathrm{sgn}(k)j}a_{r+j,s+|k|-j-1}^{1,\mathrm{sgn}(k)}\Biggr).\]
We also have
\begin{align*}
\ & d^{lk}a_{r+|k|,s-1}^2+a_{r-1,s+|k|}^3+\sum _{j=0}^{|k|-1}d^{l\mathrm{sgn}(k)j}a_{r+j, s+|k|-j-1}^{1,\mathrm{sgn}(k)} \\
   &=d^{l\mathrm{sgn}(k)(|k|-1)}(d^{l\mathrm{sgn}(k)}a_{r+|k|,s-1}^2+a_{r+|k|-2,s+1}^3+a^{1,\mathrm{sgn}(k)}_{r+|k|-1,s} \\
   &-d^{l\mathrm{sgn}(k)(|k|-1)}(a_{r+|k|-2,s+1}^3+a_{r+|k|-1,s}^2)  \\
   &+d^{l\mathrm{sgn}(k)(|k|-2)}(d^{l\mathrm{sgn}(k)}a_{r+|k|-1,s}^2+a_{r+|k|-3,s+2}^3+a_{r+|k|-2,s+1}^{1,\mathrm{sgn}(k)}) \\
   &-d^{l\mathrm{sgn}(k)(|k|-2)}(a_{r+|k|-3,s+2}^3+a_{r+|k|-2,s+1}^2)  \\
   &+\cdots -\cdots \\
   &+d^{l\mathrm{sgn}(k)}(d^{l\mathrm{sgn}(k)} a_{r+2,s+|k|-3}^2+a_{r,s+|k|-1}^3+a_{r,s+|k|-2}^{1,\mathrm{sgn}(k)}) \\
   &-d^{l\mathrm{sgn}(k)}(a_{r,s+|k|-1}^3+a_{r+1,s+|k|-2}^2 )\\
   &+(d^{l\mathrm{sgn}(k)}a_{r+1,s+|k|-2}^2+a_{r-1,s+|k|}^3+a_{r,s+|k|-1}^{1,\mathrm{sgn}(k)}).
\end{align*}
Thus we gate the conclusion.

(2) For each $k$, $j=0, \cdots , k-1$, set
\[ c_j^k:=\frac{\sqrt{3}}{3^{|k|/2}}\beta _{r+j,s+|k|-j-1}^{\mathrm{sgn}(k)}=3^{\frac{r+s}{2}}(d^{l\mathrm{sgn}(k)}a^2_{r+j+1,s+|k|-j-2}+a^3_{r+j-1,s+|k|-j}+a^{1,\mathrm{sgn}(k)}_{r+j,s+|k|-j-1}),\]
\[ d_j^k:=\frac{1}{3^{|k|/2}}\beta _{r+j,s+|k|-j}^0=3^{\frac{r+s}{2}}(a^3_{r+j-1,s+|k|-j}+a^2_{r+j,s+|k|-j-1}).\]
Then we have
\begin{align*}
\sum _{j=0}^{|k|-1}|c_j^k|^2+\sum _{j=1}^{|k|-1}|d^k_j|^2 &=3^{1-|k|}\sum _{j=0}^{|k|-1}|\beta _{r+j,s+k-j-1}^1|^2+3^{-|k|}\sum _{j=1}^{|k|-1}|\beta _{r+j,s+k-j}^0|^2 \\
                                                                                                      &\leq \frac{3}{4\cdot 3^{|k|}}\Biggl( \sum _{j=0}^{|k|-1}| \beta _{r+j,s+k-j-1}^1|^2 \|  \xi _{r+j,s+k-j-1}^1\| _2^2 +\sum _{j=1}^{|k|-1}| \beta _{r+j,s+k-j}^0|^2 \| \xi _{r+j,s+k-j}^0\| _2^2 \Biggr) \\
                                                                                                      &=\frac{3}{4\cdot 3^{|k|}}\| \sum _{j=0}^{|k|-1}\beta _{r+j,s+k-j-1}^1\xi _{r+j,s+k-j-1}^1 +\sum _{j=1}^{|k|-1}\beta _{r+j,s+k-j}^0\xi _{r+j,s+k-j}^0 \| _2^2 \\
                                                                                                      &\leq \frac{3}{4\cdot 3^{|k|}} \| \sum _{r,s,k}\beta _{r,s}^k\xi _{r,s}^k \| _2^2 \\
                                                                                                      &\leq \frac{3}{4\cdot 3^{|k|}}.
\end{align*}
Hence we have
\begin{align*}
\ &\sum _{|k|\geq k_0} |\beta _{r,s}^k |^2 \\
   &\leq \sum _{|k|\geq k_0}\Biggl( \bigl| |c^k_{k-1}|+|d^k_{k-1}|+\cdots +|c_1^k|+|d_1^k|+|c_0^k|\bigr| \Biggr) ^2 \\
   &\leq \sum _{|k|\geq k_0}\Biggl( (2|k|+1)\bigl| |c^k_{k-1}|^2+|d_{k-1}^k|^2 +\cdots +|c_1^k|^2\bigr| \Biggr) \\
   &\leq \sum _{|k|\geq k_0}\frac{3(2|k|+1)}{4\cdot 3^k} \\
   &\leq \frac{C}{2^{k_0}}.
\end{align*}
\end{proof}
Set
\[  L:=\overline{\mathrm{span}} \{ AW_l^{1, \alpha ,1}A, \ Av_i^{\pm l} A\mid i=1,2, \ l>0\} .\]
Summarizing the above results, we have the following.
\begin{lemm}
\label{3-11}
\textup{(See Lemma 3.2 of Cameron--Fang--Ravichandran--White \cite{CFRW})}
There exists a  sequence of orthonormal vectors $\{ \xi _n\}_{n=1}^\infty$ satisfying the following conditions.

\textup{(1)} Each vector $\xi _n$ lies in $W_{l(n)}^{i(n)}$ for some $l(n)\geq 1$, $i(n)=0,2, (1, \alpha ,2), (1, \beta)$.

\textup{(2)} The subspaces $\mathrm{span}A\xi _nA$ \textup{(}$n \in \mathbf{N}$ \textup{)} are pairwise orthogonal in $L^2M$.

\textup{(3)} We have 
\[ L^2M\ominus (L^2A\oplus L) =\bigoplus \overline{\mathrm{span}}A\xi _n A.\]

\textup{(4)} For any $n$ with $l(n)>1$, the sequence $\{ (\xi _n)_{r,s}/ \| (\xi _n)_{r,s} \| _2\}$ is an orthonormal basis of the subspace $\overline{\mathrm{span}}A\xi _nA$.

\textup{(5)} For each $n,m>0$, there exists a bounded invertible operator $T_{n,m}$ from the subspace $\overline{\mathrm{span}}A\xi _n^0A$ to $\overline{\mathrm{span}} A\xi _m^0A$ defined by $(\xi_n^0) _{r,s} \mapsto (\xi _m^0) _{r,s}$. 
Furthermore, there exists a constant $C_0$ satisfying $|T_{n,m}| , |T_{i,j}^{-1}|\leq C_0$ for any $n,m$.

\textup{(6)} The subspace $L$ is contained in the subspace $\bigoplus _{r,s\geq 0, \ k,l \in \mathbf{Z}, \ i=1,2}\mathbf{C}\xi _{r,s}^{i,l,k}$.
\end{lemm}
\begin{proof}
This is shown by the same argument as that of the proof of Lemma 3.2 of Cameron--Fang--Ravichandran--White \cite{CFRW}.
However, for readers' convenience, we present a proof.
For each $i=0,2,(1, \alpha ,2 ), (1, \beta)$, $l\geq 1$, choose an orthonormal basis $\{ \eta _k^{l,i}\}_{k=1}^{K_l}$ of $W_l^i\ominus S_l^i$.
Let $\{ \xi _n\} $ be a rearrangement of 
\[ \{ (\eta _k^{l,i})\}_{l\geq 1, \ i=0,2,(1,\alpha ,2 ), (1, \beta) \ k=1, \cdots , K_l}.\]
Then by construction, the sequence $\{ \xi _n\}$ satisfies condition (1).
By Lemma \ref{3-9} (2), the sequence $\{ \xi _n\}$ satisfies condition (2).

We show that $\{ \xi _n \}$ satisfies condition (4).
By Lemma \ref{3-7}, the set $\{ (\xi _n)_{r,s}\}$ spans $A\xi_n A$.
By condition (2) and Lemma \ref{3-8}, the vectors $\{  ( \xi _n)_{r,s}\}_{n,r,s} $ are mutually orthogonal if $l(n)>1$.
Thus we have condition (4).

Next, we show that the sequence $\{ \xi _n\}$ satisfies condition (3).
By Lemma \ref{3-9} (3), $L$ is orthogonal to $\bigvee \overline{\mathrm{span}}A\xi _nA$.
By Lemma \ref{3-9} (2), the subspaces $A\xi _nA$'s are mutually orthogonal.
Thus it is enough to show that the set $\{ A\xi _nA\}_n$ really spans $L^2M\ominus (L^2A\oplus L)$.
Take an element $\xi \in W_l\ominus ( L^2A \oplus L)$ which is orthogonal to any $A\xi _nA$.
Then $\xi$ is orthogonal to the space $W_l\ominus S_l$, which means that $\xi \in S_l$.
On the other hand, by the same argument as in the proof of Lemma \ref{3-9} (1), any vector $w\in W_k$  is written as a linear combination of the form $\chi _n \gamma \chi _m$ for some $n,m\geq 0$, $\gamma \in  W_{k'}\ominus S_{k'}$ ($k'\leq k$).
Thus the vector $\xi $ is orthogonal to $AwA$ for any $w\in W_k$, $k\leq l-1$.
Hence by Lemma \ref{3-9} (1), the vector $\xi $ is orthogonal to $S_l$. 
Thus the vector $\xi$ is zero.

Condition (5) follows in the same way as the proof of Lemma 3.2 of Cameron--Fang--Ravichandran--White \cite{CFRW}.
Condition (6) is trivial.
\end{proof}

By an immediate consequence of Lemma \ref{3-11}, we have the following.

\begin{cor}
\label{cor-3}
We have the following.

\textup{(1)} The family $\{ (\xi _m)_{r,s}\mid m \in \mathbf{N}, r,s \geq 0\}$ is a Riesz basis of $L^2M\ominus (A\oplus L)$.

\textup{(2)} We have 
\[ L\subset \bigoplus _{i,l,k,r,s}\xi _{r,s}^{i,l,k}.\]

\textup{3)} We have $L\oplus (L^2M\ominus (A\oplus L))=L^2M\ominus A$.
\end{cor}

\section{Locating the support of the sequences of $(M^\omega\ominus A^\omega )\cap A'$}
Now, we would like to explain how to show the maximal amenability of the subalgebra.
In order to show the maximal amenability, we look at the following notion.
\begin{defi}
\textup{(Lemma 2.1 of Popa \cite{P})}
Let $M$ be a factor of type $\mathrm{II}_1$ and $A$ be a von Neumann subalgebra of $M$.
We say that the subalgebra $A$ has the asymptotic orthogonality property if for any $x^1=(x_n^1), x^2=(x_n^2)\in (M^\omega \ominus A^\omega )\cap A'$, any $y_1,y_2\in M \ominus A$, we have $\tau ^\omega (y_1^*{x^1}^*y_2x^2)=0$.
\end{defi}
Although the definition of the asymptotic orthogonality property is rather technical, this property is crucial because of the following proposition.
\begin{prop}
\textup{(Corollary 2.3 of Cameron--Fang--Ravichandran--White \cite{CFRW}, See also Popa \cite{P})}
\label{6-13}
Let $A$ be a singular maximal abelian subalgebra of a factor  $M$ of type $\mathrm{II}_1$ with the asymptotic orthogonality property.
Then it is maximal amenable.
\end{prop}
This is why we would like to show that the subalgebra has the asymptotic orthogonality property.
In order to achieve this, we first show that any sequence of $(M^\omega \ominus A^\omega )\cap A'$ eventually get out of the space  $\mathrm{span}\{ \xi _{r,s}^{i,l,k}, (\xi _n)_{r,s}\mid r\leq M \ \mathrm{or} \ s\leq M\}$ for any $M>0$.
Then we show that any vectors $\eta _1, \eta _2 $ orthogonal to the space and any vector $a,b, \in M\ominus A$,  the value $|\tau (a*\eta _1^*b\eta _2)|$ is small if $M$ is large enough.
In this section, we show the first part.

As in Section 3, set
\[ \tilde{\chi }_1 :=\frac{1}{\sqrt{3}} (u_1+u_1^{-1}+u_2+u_2^{-1}).\]
Then we have the following lemma.
\begin{lemm}
\textup{(See Lemma 4.3 of Cameron--Fang--Ravichandran--White \cite{CFRW}, See also Lemma 11 of Wen \cite{W})}
\label{5-0}
Let $D\subset A$ be a diffuse von Neumann subalgebra.
Let $x=(x_n) $ be an $\omega$-centralizing sequence of $M$ commuting with $D$, $\| x_n \| =1 $ and $E_{A}(x_n)=0$ for all $n$.
Assume that each $x_n$ is written as $x_n =\sum _{m,r,s} \alpha _{r,s}^{n,m}(\xi _m)_{r,s}$ for some $\alpha _{r,s}^{n,m}\in \mathbf{C}$.
Then for each $M\in \mathbf{N}$, we have
\[ \lim _{n \to \omega}\sum _{ m\geq 1, r\leq M \ \mathrm{or} \ s\leq M}\| \alpha _{r,s}^{n,m}\| _2^2=0.\]
\end{lemm}
\begin{proof}
This is shown by the same way as that of Lemma 4.3 of Cameron--Fang--Ravichandran--White \cite{CFRW}.
\end{proof}

  By Lemma \ref{5-1}, we have
   \begin{align*}
\ & \tilde{\chi}_1 \xi _{r,s}^k-\xi _{r,s}^k \tilde{\chi }_1 \\
    &=\begin{cases}
    \xi _{r+1,s}^k +\xi _{r-1,s}^k-\xi _{r,s+1}^k -\xi _{r,s-1}^k  & (r, s \geq 1) \\
    \xi _{1,s}^k +\frac{1}{\sqrt{3}}(d^l\xi _{0,s}^{k+1}+d^{-l}\xi _{0,s}^{k-1})-\xi _{0,s+1}^k -\xi _{0,s-1}^k  & (r=0, s \geq 1) \\
   \xi _{r+1,0}^k +\xi _{r-1,0}^k - \xi _{r,1}^k -\frac{1}{\sqrt{3}}(\xi _{r,0}^{k+1}+\xi _{r,0}^{k-1}) & (r\geq 1, s=0) \\
    \xi _{1,s}^k +\frac{1}{\sqrt{3}}(d^l\xi _{0,s}^{k+1}+d^{-l}\xi _{0,s}^{k-1})-\xi _{r,1}^k -\frac{1}{\sqrt{3}}(\xi _{r,0}^{k+1}+\xi _{r,0}^{k-1}) & (r=s=0) .
  \end{cases}
  \end{align*}
  Hence for $x=\sum _{i,l,r,s,k} \alpha _{r,s}^{i,l,k} \xi _{r,s}^{i,l,k}$, where $\alpha _{r,s}^{i,l,k} \in \mathbf{C}$ for each $r,s,i,l,k$,  write
  \[ \tilde{\chi}_1x-x\tilde{\chi}_1=\sum _{r,s,i,l,k}\beta _{r,s}^{i,l,k}\xi _{r,s}^{i,l,k}.\]
  Then the complex number $\beta _{r,s}^k$ is the following.
  \[ \beta_{r,s}^{i,l,k}=\begin{cases}
  \alpha _{r+1,s}^{i,l,k} +\alpha _{r-1,s}^{i,l,k} -\alpha _{r,s+1}^{i,l,k}-\alpha _{r,s-1}^{i,l,k} &(r,s\geq 1) \\
  \alpha _{1,s}^{i,l,k }-\alpha _{0,s-1}^{i,l,k} -\alpha _{0,s+1}^{i,l,k} +\frac{1}{\sqrt{3}} (d^l\alpha _{0,s}^{i,l,k-1}+d^{-l}\alpha _{0,s}^{i,l,k+1} ) & (r=0, s\geq 1) \\
  \alpha _{r-1,0}^{i,l,k}+\alpha _{r+1,0}^{i,l,k} -\alpha _{r,1}^{i,l,k} -\frac{1}{\sqrt{3}}(\alpha _{r,0}^{i,l,k-1}+\alpha _{r,0}^{i,l,k+1}) & (r\geq 1, s=0) \\
  \alpha _{1,0}^{i,l,k}-\alpha _{0,1}^{i,l,k} +\frac{1}{\sqrt{3}} (d^l-1)(\alpha _{0,0}^{i,l,k-1}-d^{-l}\alpha _{0,0}^{i,l,k+1}) &(r=s=0).
  \end{cases}
  \]
  \begin{lemm}
  \label{5-2}
  For $x=\sum \alpha _{r,s}^{i,l,k}\xi _{r,s}^{i,l,k}$ with $\| x\| _2=1$, $s'\geq s\geq 1$, we have the following inequalities.
  \begin{align*}
  \Biggl( \sum _{r\geq s', i,l,k} &| \alpha _{r-s,0}^{i,l,k} +\alpha _{r-s+2,0}^{i,l,k} +\cdots +\alpha _{r+s,0}^{i,l,k} -\frac{1}{\sqrt{3}}\left( \alpha _{r-s+1,0}^{i,l,k-1}+\cdots + \alpha _{r+s-1, 0}^{i,l,k-1} \right) \\
                                                                      &-\frac{1}{\sqrt{3}}\left( \alpha _{r-s+1,0}^{i,l,k+1}+\cdots + \alpha _{r+s-1,0}^{i,l,k+1}\right) |^2  \biggr) ^{1/2} -\Biggl( \sum _{r\geq s' , i,l,k} |\alpha _{r,s}^{i,l,k }|^2 \Biggr) ^{1/2} \\
                                                                      &\leq \Biggl( \sum _{r\geq s, i,l,k} | \alpha _{r,s}^{i,l,k}  -\left(\alpha _{r-s,0}^{i,l,k}+\cdots + \alpha _{r+s,0}^{i,l,k} \right) \\
                                                                      &+\frac{1}{\sqrt{3}}\left( \alpha _{r-s+1,0}^{i,l,k-1}+\cdots +\alpha _{r+s-1,0}^{i,l,k-1}\right) +\frac{1}{\sqrt{3}} \bigl( \alpha _{r-s+1,0}^{i,l,k+1}+\cdots + \alpha _{r+s-1,0}^{i,l,k+1}\bigr) |^2\Biggr) ^{1/2} \\
                                                                      &\leq 3^{s-1}C_0 \| [x, \tilde{\chi}_1]\| _2.
\end{align*}
  \end{lemm}
  \begin{proof}
 This is shown in a similar way to Lemma 4.1 of Cameron--Fang--Ravichandran--White \cite{CFRW}.
  \end{proof}
 Similarly, we have the following.
 \begin{lemm}
 \label{5-3}
 For $x=\sum \alpha _{,s}^{i,l,k}\xi _{r,s}^{i,l,k}$ with $\| x\| _2=1$, $r'\geq r\geq 1$, we have the following inequalities.
  \begin{align*}
  \Biggl( \sum _{s\geq r', i,l,k} &| \alpha _{0,s-r}^{i,l,k} +\alpha _{0,s-r+2}^{i,l,k} +\cdots +\alpha _{0,s+r}^{i,l,k} -\frac{d^l}{\sqrt{3}}\left( \alpha _{0, s-r+1}^{i,l,k-1}+\cdots + \alpha _{0,s+r-1}^{i,l,k-1} \right) \\
                                                                      &-\frac{d^{-l}}{\sqrt{3}}\left( \alpha _{0,s-r+1}^{i,l,k+1}+\cdots + \alpha _{0,s+r-1}^{i,l,k+1}\right) |^2  \biggr) ^{1/2} -\Biggl( \sum _{s\geq r' , i,l,k} |\alpha _{r,s}^{i,l,k} |^2 \Biggr) ^{1/2} \\
                                                                      &\leq \Biggl( \sum _{s\geq r, i,l,k} | \alpha _{r,s}^{i,l,k}  -\left(\alpha _{0,s-r}^{i,l,k}+\cdots + \alpha _{0,s+r}^{i,l,k} \right) \\
                                                                      &+\frac{d^l}{\sqrt{3}}\left( \alpha _{0,s-r+1}^{i,l,k-1}+\cdots +\alpha _{0,s+r-1}^{i,l,k-1}\right) +\frac{d^{-l}}{\sqrt{3}} \bigl( \alpha _{0,s-r+1}^{i,l,k+1}+\cdots + \alpha _{0, s+r-1}^{i,l,k+1}\bigr) |^2\Biggr) ^{1/2} \\
                                                                      &\leq 3^{r-1}C_0 \| [x, \tilde{\chi}_1]\| _2.
\end{align*}
 \end{lemm}
 \begin{lemm}
 \label{5-4}
 For $x^n =\sum _{r,s,i,l,k}\alpha _{r,s}^{n,i,l,k} \xi _{r,s}^{i,l,k}$ with $\|x^n\|_2 =1$ and $\| [x^n, \tilde{\chi}_1]\|_2\to 0$ as $n\to \omega$, we have
 \[ \lim _{n \to \omega }\sum _{r\geq 0, \ i,l,k}|\alpha _{r,0}^{n,i,l,k}-\frac{1}{\sqrt{3}}(\alpha _{r+1,0}^{n,i,l,k-1}+\alpha _{r+1,0}^{n,i,l,k+1})|^2=0,\]
 \[ \lim _{n\to \omega } \sum _{s\geq 0, \ i,l,k}|\alpha _{0,s}^{n,i,l,k}-\frac{1}{\sqrt{3}}(d^l\alpha _{0,s+1}^{n,i,l,k-1}+d^{-l}\alpha _{0,s+1}^{n,i,l,k+1})|^2=0.\]
 \end{lemm}
 \begin{proof}
 This is shown in a similar way to that in Lemma 4.2 of Cameron--Fang--Ravichandran--White \cite{CFRW}.
 \end{proof}
 \begin{lemm}
 \label{5-5}
 For $x^n =\sum _{r,s,i,l,k}\alpha _{r,s}^{n,i,l,k} \xi _{r,s}^{i,l,k}\in L$ with $\|x^n\|_2 =1$ and $\| [x^n, \tilde{\chi}_1]\|_2\to 0$ as $n\to \omega$, we have
 \[ \lim _{n \to \omega }\sum _{i,l,k,s}|\alpha _{0,s}^{n,i,l,k}|^2=0,\]
 \[ \lim _{n \to \omega }\sum _{i,l,k,r}|\alpha _{r,0}^{n,i,l,k}|^2=0.\]
 \end{lemm}
 \begin{proof}
By the middle $\leq $ the right of Lemma \ref{5-2} and Lemma \ref{5-4}, for any $\epsilon >0$ and $s\in \mathbf{N}$, there exists a natural number $N_s$ such that
\[ \Biggl( \sum _{r\geq s, i,l,k } |\alpha _{r,s}^{n,i,l,k}-\alpha _{r+s,0}^{n,i,l,k}|^2 \Biggr) ^{1/2}<\epsilon \tag{*} \]
for any $n \geq N_s$.
Similarly, by Lemmas \ref{5-3} and \ref{5-4}, for any $\epsilon >0$ and $r\in \mathbf{N}$, there exists a natural number $N_r$ such that
\[ \Biggl( \sum _{s\geq r, i,l,k} |\alpha _{r,s}^{n,i,l,k}-\alpha _{0,s+r}^{n,i,l,k}|^2 \Biggr) ^{1/2}<\epsilon \tag{**} \]
for any $n \geq N_r$.
Fix a natural number $s_0$.
By picking up terms over $r\geq 2s_0-s$ of the first inequality for $s=1, \cdots s_0$, we have
\[ \Biggl( \sum _{r\geq 2s_0-s, \ i,l,k} |\alpha _{r,s}^{n,i,l,k}-\alpha _{r+s,0}^{n,i,l,k}|^2\Biggr) ^{1/2}<\epsilon \]
for $n\geq N_1, \cdots , N_{s_0}$.
By the triangle inequality, we have
\[ \left| \Biggl( \sum _{r\geq 2s_0-s, \ i,l,k}|\alpha _{r,s}^{n,i,l,k}|^2\Biggr) ^{1/2}-\Biggl( \sum _{r\geq 2s_0, i,l,k}|\alpha _{r,0}^{n,i,l,k}|^2\Biggr) ^{1/2}\right| <\epsilon .\]
Here, we re-enumerated the index of the second term of the left hand side of the above inequality.
Hence we have
\[ \Biggl( \sum _{r\geq 2s_0, \ i,l,k}|\alpha _{r,0}^{n,i,l,k}|^2 \Biggr) ^{1/2} < \Biggl( \sum _{r\geq 2s_0-s, \ i,l,k}|\alpha _{r,s}^{n,i,l,k}|^2\Biggr) ^{1/2}+\epsilon .\]
Taking the square of the above inequality, we have
\begin{align*}
 \sum _{r\geq 2s_0, \ i,l,k} |\alpha _{r,0}^{n,i,l,k}|^2 &<\sum _{r\geq 2s_0-s, \ i,l,k}|\alpha _{r,s}^{n,i,l,k}|^2+2\Biggl( \sum _{r\geq 2s_0-s, \ i,l,k}|\alpha _{r,s}^{n,i,l,k}|^2\Biggr) ^{1/2}\epsilon +\epsilon ^2 \\
                                                                              &\leq \sum _{r\geq 2s_0-s, \ i,l,k}|\alpha _{r,s}^{n,i,l,k}|^2+ 2C_0 \epsilon +\epsilon ^2.
 \end{align*}
 Taking the average of the above inequality over $s=1, \cdots ,s_0$, we have
 \begin{align*}
 \sum _{r\geq 2s_0, \ i,l,k} |\alpha _{r,0}^{n,i,l,k}|^2 &< \frac{1}{s_0}\Biggl( \sum _{r\geq 2s_0-s, \ s=1, \cdots , s_0, \ i,l,k}|\alpha _{r,s}^{n,i,l,k}|^2\Biggr) +2C_0\epsilon +\epsilon ^2 \\
                                                                             &<\frac{1}{s_0}C_0^2+2C_0\epsilon +\epsilon ^2.
 \end{align*}
 Hence we have 
 \[ \lim _{n \to \omega } \Biggl( \sum _{r\geq 2s_0, \ i,l,k}|\alpha _{r,0}^{n,i,l,k}|^2\Biggr) ^{1/2}\leq \frac{C_0}{\sqrt{s_0}}.\]
 Similarly, we have
 \[ \lim _{n \to \omega } \Biggl( \sum _{s\geq 2r_0, \ i,l,k}|\alpha _{0,s}^{n,i,l,k}|^2\Biggr) ^{1/2}\leq \frac{C_0}{\sqrt{s_0}}.\]
 On the other hand, by Lemma \ref{16} (2), for any $k_0$, we have
 \[ \lim _{n \to \omega} \Biggl( \sum _{r<2s_0, \ |k|\geq k_0, i,l}|\alpha _{r,0}^{n,i,l,k}|^2\Biggr) ^{1/2}<\frac{C}{2^{k_0}}.\]
 Next, by looking at the partial sum over $r=s(\geq 1)$ of  inequality (*), for $n\geq N_r$, we have
 \[ \Biggl( \sum _{i,l,k} |\alpha _{r,r}^{n,i,l,k}-\alpha _{2r,0}^{n,i,l,k}|^2\Biggr) ^{1/2}<\epsilon .\]
 By looking at the partial sum over $r=s$ of inequality (**), there exists a natural number $M_r$ such that for any $n \geq M_r$, we have
 \[ \Biggl( \sum _{i,l,k} |\alpha _{r,r}^{n,i,l,k}-\alpha _{0,2r}^{n,i,l,k}|^2\Biggr) ^{1/2}<\epsilon . \]
 Hence for $n \geq N_r, M_r$, we have
 \[ \Biggl( \sum _{i,l,k} |\alpha _{0,2r}^{n,i,l,k}-\alpha _{2r,0}^{n,i,l,k}|^2\Biggr) ^{1/2}<2\epsilon  \tag{***}.\]
 By using this and Lemma \ref{5-4}, and by taking $\epsilon $ smaller,  we have
 \[ \Biggl( \sum _{i,l,k} |\alpha _{0,2r-1}^{n,i,l,k}-\alpha _{2r-1,0}^{n,i,l,k}|^2\Biggr) ^{1/2}<2\epsilon \tag{****}.\]
Suppose that there exited $r\geq 2$ with $\lim _{n \to \omega } \sum _{i,l}| \alpha _{r,0}^{n,i,l,0}|^2=c>0$.
Take a small positive number $\delta >0$, which depends on $c$ and is determined later.
Since $x^n \in L$, by Lemma \ref{16} (1),  we have
\[ \alpha _{r,s}^{n,i,l,k}=\frac{1}{3^{|k|/2}}(\sqrt{3}\sum _{j=0}^{|k|-1}d^{l\mathrm{sgn}(k)j}\alpha ^{n,i,l,\mathrm{sgn}(k)}_{r+j,s+|k|-j-1}-\sum _{j=1}^{|k|-1}d^{l\mathrm{sgn}(k)j}\alpha ^{n,i,l,0}_{r+j,s+|k|-j})\]
for any $i,l,k,r,s$.
On the other hand, by Lemma \ref{5-4}, we have
 \[ \sum _{r\geq 0, \ i,l,k}|\alpha _{r,0}^{n,i,l,k}-\frac{1}{\sqrt{3}}(\alpha _{r+1,0}^{n,i,l,k-1}+\alpha _{r+1,0}^{n,i,l,k+1})|^2<\delta \]
for any sufficiently large $n$.
By inequalities (*), (**), (***) and (****), for any fixed $(r,s)\in \mathbf{N}^2$, we have
\[ (\sum _{i,l,k}|\alpha ^{n,i,l,k}_{r,s}-\alpha ^{n,i,l,k}_{r+s,0}|^2)^{1/2}<\delta \]
for any sufficiently large $n$.
From now, for any fixed $(r,k)$, we regard $\{ \alpha ^{n,i,l,k}_{r,0}\}_{i,l}$ as a vector of $\ell ^2 (\{ 1,2\}\times \mathbf{Z})$.
Then we have
\[ \| \alpha ^{n,i,l,k}_{r,0}-\frac{1}{3^{|k|/2}}(\sqrt{3}\sum _{j=0}^{|k|-1}d^{l\mathrm{sgn}(k)j}\alpha _{r+|k|-1,0}^{n,i,l,\mathrm{sgn}(k)}-\sum _{j=1}^{|k|-1} d^{l\mathrm{sgn}(k)j} \alpha _{r+|k|,0}^{n,i,l,0})\| _2< \frac{|k|}{3^{\frac{|k|-1}{2}}}\cdot 2\delta ,\]
\[ \| \alpha ^{n,i,l,k+\mathrm{sgn}(k)}_{r+1,0}-\frac{1}{3^{\frac{|k|+1}{2}}}(\sqrt{3}\sum _{j=0}^{|k|} d^{l\mathrm{sgn}(k)j}\alpha ^{n,i,l,\mathrm{sgn}(k)}_{r+|k|+1,0}-\sum _{j=1}^{|k|} d^{l\mathrm{sgn}(k)j}\alpha ^{n,i,l,0}_{r+|k|+2,0})\| _2 <\frac{|k|+1}{3^{\frac{|k|}{2}}}  \cdot 2\delta ,\] 
\[ \| \alpha ^{n,i,l,k-\mathrm{sgn}(k)}_{r+1,0}-\frac{1}{3^{\frac{|k|-1}{2}}}(\sqrt{3}\sum _{j=0}^{|k|-2} d^{l\mathrm{sgn}(k)j}\alpha ^{n,i,l,\mathrm{sgn}(k)}_{r+|k|-1,0}-\sum _{j=1}^{|k|-2}d^{l\mathrm{sgn}(k)j}\alpha ^{n,i,l,0}_{r+|k|,0})\| _2 <\frac{|k|-1}{3^{\frac{|k|+1}{2}}}\cdot 2\delta .\]
Thus we have
\begin{align*}
\delta &> \| \alpha ^{n,i,l,k}_{r,0}-\frac{1}{\sqrt{3}}(\alpha ^{n,i,l,k+\mathrm{sgn}(k)}_{r+1,0}+\alpha ^{n,i,l,k-\mathrm{k}}_{r+1,0}) \| _2 \\
         &>-\frac{|k|+1}{3^{\frac{|k|-1}{2}}} (1+\frac{2}{\sqrt{3}})\cdot 2\delta +\frac{1}{3^{\frac{|k|}{2}}}\| (\sqrt{3}\sum _{j=0}^{|k|-1}d^{l\mathrm{sgn}(k)j}\alpha _{r+|k|-1,0}^{n,i,l,\mathrm{sgn}(k)}-\sum _{j=1}^{|k|-1}d^{l\mathrm{sgn}(k)j}\alpha _{r+|k|,0}^{n,i,l,0})   \\
         &-\frac{1}{\sqrt{3}^2} (\sqrt{3}\sum _{j=0}^{|k |} d^{l\mathrm{sgn}(k)j}\alpha ^{n,i,l,\mathrm{sgn}(k)}_{r+|k|+1,0}-\sum _{j=1}^{|k|} d^{l\mathrm{sgn}(k)j}\alpha ^{n,i,l,0}_{r+|k|+2,0}) \\
         &- \frac{\sqrt{3}}{\sqrt{3}}(\sqrt{3}\sum _{j=0}^{|k|-2} d^{l\mathrm{sgn}(k)j}\alpha ^{n,i,l,\mathrm{sgn}(k)}_{r+|k|-1,0}-\sum _{j=1}^{|k|-2}d^{l\mathrm{sgn}(k)j}\alpha ^{n,i,l,0}_{r+|k|,0}) \| _2\\ 
         &>\frac{1}{3^{\frac{|k|}{2}}}\| (\sqrt{3}\alpha _{r+|k|-1,0}^{n,i,l,\mathrm{sgn}(k)}-\alpha _{r+|k|,0}^{n,i,l,0}) \\
         &-\frac{1}{3}(\sqrt{3}\sum _{j=-|k|+1}^1 d^{l\mathrm{sgn}(k)j}\alpha ^{n,i,l,\mathrm{sgn}(k)}_{r+|k|+1,0} -\sum _{j=-|k|+2}^1 d^{l\mathrm{sgn}(k)j}\alpha ^{n,i,l,0}_{r+|k|+2,0}) \| _2 -6\frac{|k|+1}{3^{\frac{|k|-1}{2}}}\delta \\
          &> \mathrm{(A)}_{r,k}-999\delta 
\end{align*}
for any sufficiently large $n$ (How large we should take $n$ depends on $r$, $k$ and $\delta$).
Thus we have $\mathrm{(A)}_{r,k}<1000\delta $.
Now, we have $r+|k|=(r+1)-(|k|-1)$.
Hence for a fixed $r\geq 3$, we have
\begin{align*}
\| \frac{1}{3}(\sqrt{3}\alpha ^{n,i,l,\mathrm{sgn}(k)}_{r,0}-d^{l\mathrm{sgn}(k)}\alpha ^{n,i,l,0}_{r+1,0}) \| _2 &\leq \| \mathrm{(A)}_{r-|k|-1,|k|}-\mathrm{(A)}_{r-|k|,|k|-1}\| _2 \\
                                                                                                                                                                 &< 2000\delta .
\end{align*}
Thus we have
\[ \| \sqrt{3}\alpha ^{n,i,l,\mathrm{sgn}(k)}_{r,0}-d^{l\mathrm{sgn}(k)}\alpha _{r+1,0}^{n,i,l,0}\| _2 \leq 10000\delta .\]
Hence we have
\[ \| \alpha _{r,0}^{n,i,l,1}+\alpha _{r,0}^{n,i,l,-1}-\frac{1}{\sqrt{3}}(d^l+d^{-l})\alpha _{r+1,0}^{n,i,l,0}\| <19990\delta .\]
Thus we have 
\[ \| \sqrt{3}\alpha _{r-1,0}^{n,i,l,0}-\frac{1}{\sqrt{3}}(d^{l}+d^{-l})\alpha _{r+1,0}^{n,i,l,0}\| <20000\delta \]
for any sufficiently large $n$ (depending on $r$ and $\delta $).
Hence if we take $\delta $ so large that it satisfies $\delta <c/100000 $, we have
\[ \| \alpha _{r+1}^{n,i,l,0}\| _2 \geq \frac{4}{3}c \]
if we take a large $n$.
Take a large $T\in \mathbf{N}$ so large that it satisfies $T>1/c$.
Then,  by induction, for any $T\geq t>0$, we have
\[  \| \alpha _{r+2t-1, 0}^{n,i,l,0}\| _2 \geq \frac{4^t}{3^t}c \]
for any large $n$ (depending on $r$, $T$ and $\delta$), which would contradict the fact that $ \sum _{i,l,k,r,s}| \alpha ^{n,i,l,k}_{r,s}|^2 \leq 1$.
Hence we have $ \| \alpha ^{n,i,l,0}_{r,0}\|_2 \to 0$ for any $r\geq 2$.
By using other inequalities, it is possible to show that $\| \alpha ^{n,i,l,k}_{r,0} \| _2 \to 0$ for any $r,k$.

Hence we have
\begin{align*}
\ & \lim _{n \to \omega} \sum _{i,l,k}|\alpha _{r,0}^{n,i,l,k}|^2 \\
   &=\lim _{n \to \omega } \Biggl( \sum _{r\geq 2s_0, \ i,l,k}+\sum _{r<2s_0, \ |k|\geq k_0}+\sum _{r<2s_0, \ |k|<k_0}\Biggr) |\alpha _{r,0}^{n,i,l,k}|^2 \\
   &\leq \frac{C_0^2}{s_0}+\frac{C}{2^{k_0}}+0.
\end{align*}
For any $\epsilon >0$, we choose $s_0$ so large that we have $C_0^2/s_0<\epsilon /2$ and then we choose $k_0$ so huge that we have $C/2^{k_0}<\epsilon /2$.
Then we have 
\[ \sum _{r\geq 0, \ i,l,k}|\alpha _{0,r}^{n,i,l,k}|^2<\epsilon .\]
Thus we have $\lim _{n \to \omega } \sum _{r\geq 0, \ i,l,k}|\alpha _{0,r}^{n,i,l,k}|^2=0$.
By the same argument as above, we have
\[ \lim _{n \to \omega} \sum _{i,l,k}|\alpha _{0,s}^{n,i,l,k}|^2=0.\]
Thus we are done.
\end{proof}
\section{Counting the number of words which contribute to the inner product}
In this section, we show that for any vectors $\eta _1, \eta _2\in \mathrm{span}\{ \xi _{r,s}^{i,l,k}, (\xi _m)_{r,s}\mid r\geq M \ \mathrm{or} \ s\geq M\}$, any vectors $a,b \in M\ominus A$, the inner product $|\tau (a^*\eta _1^*b\eta _2)|$ is small if $M$ is large enough (Lemmas \ref{6-1} and \ref{6-2}).
This section corresponds to Section 5 and Lemma 6.1 of Cameron--Fang--Ravichandran--White \cite{CFRW}.
\begin{lemm}
\label{6-2}
Let $g, h$ be elements of $M \ominus A$ satisfying the following conditions.

\bigskip

\textup{(1)} The vector $g$ is of the form $u_{i_1}^{k_1}w_{i_1}\cdots u_{i_n}^{k_n}w_{i_n}$ for some $n\geq 1$, $i_1\not =\cdots \not =i_n$, $k_1, \cdots , k_n\in \mathbf{Z}$, where for any $s=1, \cdots , n$, the operator $u_{i_{s}}^{k_{s}}w_{i_{s}}$ satisfies either \textup{(a)} $k_{s}\not =0$, $w_{i_{s}}=1$ or \textup{(b)} $w_{i_{s}}$ is a normalizing unitary of $\{ u_{i_{s}}\}''$ which is orthogonal to $\{ u_{i_{s}}\}''$.

\textup{(2)} The vector $h$ is of the form $u_{i'_1}^{k'_1}x_{i'_1}\cdots u_{i'_{n'}}^{k'_{n}}x_{i'_{n'}}$ for some $n'\geq 1$, $i'_1\not =\cdots \not =i'_{n'}$, $k'_1, \cdots , k'_n\in \mathbf{Z}$, where for any $t=1, \cdots , n'$, the operator $u_{i'_{t}}^{k'_{t}}x_{i'_{t}}$ satisfies either \textup{(c)} $k'_{t}\not =0$, $x_{i'_{t}}=1$ or \textup{(d)} $x_{i'_{t}}$ is a normalizing unitary of $\{ u_{i'_{t}}\}''$ which is orthogonal to $\{ u_{i'_{t}}\}''$.

\textup{(3)} At least one of the operators $u_{i_s}^{k_s}w_{i_s}$ \textup{ (}$s=1, \cdots , n$\textup{)}, $u_{i'_{t}}^{k'_{t}}x_{i'_{t}}$\textup{ (}$t=1, \cdots , n'$\textup{)} satisfies the above condition \textup{(b)} or \textup {(d)}.

\bigskip

Then there exists a positive constant $C>0$, which depends neither on $w_{i_s}$'s nor $x_{i'_t}$'s, such that for any $M >4k:=4(|k_1|+|k_2|+\cdots +|k_n|+|k'_1|+\cdots +|k'_{n'}|)$, any vectors $\eta _1, \eta _2 $ of  the space
\begin{align*}
\ &\mathrm{span}\{ (\xi _m)_{2M+r, 2M+s}\mid  m\geq 0 , \ r,s\geq 0\} \\
   &\vee \Biggl( \mathrm{span}\{ \xi ^{i,l,k}_{2M+r, 2M+s}\mid i=1,2, \ l\in \mathbf{Z}, \ k\in \mathbf{Z} , r,s\geq 0\} \cap L\Biggr) ,
\end{align*}
 we have 
\[ | \langle \eta _1 g, h \eta _2 \rangle | \leq C\frac{M^4}{3^{M/2}}\| \eta _1 \| _2 \| \eta _2 \| _2.\]
\end{lemm}
Before proving the lemma, we have to notice that the above $k_s$'s,  $k'_t$'s, $w_{i_s}$'s and $x_{i'_t}$'s are not completely determined by $h$ and $g$.
This lemma means that the above equation holds for any vectors $h$ and $g$ which admit the above presentation.
\begin{proof}
Roughly speaking, the strategy is the following.
Decompose the vectors $\eta _1^*$ and $\eta _2 $ into linear combinations of $wyw'$, $w''zw'''$, where $w,w',w'',w''', y,z $ are words with $|w|, |w'|,|w''|,|w'''|\gg |g|,|h|$.
Then we cancel each two neighboring words as possible.
The vital point is that at least one of $h^*$ and $g$ is not completely canceled by condition (3).
By using this fact, we show that for most $ (w,w',w'',w''')$, we have $\tau (h^*wyw'gw''zw''')=0$ (Claim 3).

Since the two spaces 
\[ \mathrm{span}\{ (\xi _m^i)_{2M+r, 2M+s}\mid  m\geq 0 , \ r,s\geq 0\}\]
and
\[ \mathrm{span}\{ \xi ^{i,l,k}_{2M+r, 2M+s}\mid i=1,2, \ l\in \mathbf{Z}, \ k\in \mathbf{Z} , r,s\geq 0\} \cap L\]
are mutually orthogonal, it is enough to show the following claim.

\bigskip

\textbf{Claim 1.}
Assume that vectors $\eta _1$ and $\eta _2$ belong to one of the above two spaces.
Then we have
\[ | \langle \eta _1 g, h \eta _2 \rangle | \leq C\frac{M^4}{3^{M/2}}\| \eta _1 \| _2 \| \eta _2 \| _2.\]

\bigskip

In the rest of this proof, we devote our attention to proving this claim.
For simplicity, we assume that both $\eta _1$ and $\eta _2$ belong to the former subspace (We can handle the other three cases in the same way).
Write $\eta _1$ and $\eta _2^*$ in the following way.
\[ \eta _1=\sum _{r,s\geq 0, \ m\in \mathbf{N}} 3^{-\frac{4M+r+s}{2}}\lambda _{m, 2M+r,2M+s} (\xi _m) _{2M+r,2M+s} ,\]
\[ \eta _2^*= \sum _{r',s'\geq 0, \ m' \in \mathbf{N}} 3^{-\frac{4M+r'+s'}{2}}\mu _{m', 2M+r',2M+s'}(\xi _{m'})_{2M+r',2M+s'} .\]
Then by Lemma \ref{3-11} (5), we have
\[ \| \eta _1\| _2 \geq C_0^{-1}\Biggl( \sum _{r,s,m}|\lambda _{m,2M+r,2M+s}|^2\Biggr) ^{1/2},\]
\[ \| \eta _2\| _2\geq C_0^{-1} \Biggl( \sum _{r',s',m'}|\mu _{m',2M+r',2M+s'}|^2\Biggr) ^{1/2}.\]

Set two vectors $h_{j''',j}$ and $g_{j',j''}$ in the following way.
When $w_{i_n}=1$, set
\[ h^*_{j''',j}:=u_{i_1}^{k_1-j'''}w_{i_1}u_{i_2}^{k_2}\cdots w_{i_{n-1}}u_{i_n}^{k_n-j}.\]
Here, $j'''$ and $j$ run over the following ranges.
When $n \geq 2$, $(j''',j)$ run over all $j'''=0, \cdots , k_1$, $j=0, \cdots ,k_n$.
When $n=1$, $k_1 \geq 0$, $(j''',j)$ run over all  $j'''=0, \cdots , k_1$, $j=0, \cdots , k_1$, $0\leq j+j'''\leq k_1$.
When $n=1$, $k_1<0$, $(j''',j)$ run over all $j'''=0, \cdots , k_1$, $j=0, \cdots , k_1$, $0\geq j+j'''\geq k_1$.

When $w_{i_n}\not =1$, set 
\[ h^*_{j''',j}:=u_{i_1}^{k_1-j'''}w_{i_1}\cdots u_{i_n}^{k_n}w_{i_n}\]
for $j'''=0, \cdots, k_1$, $j=0$.

When $x_{i'_{n'}}=1$, set
\[ g_{j',j''}:=u_{i'_1}^{k'_1-j'}x_{i'_1}u_{i'_2}^{k'_2}\cdots u_{i'_{n'}}^{k'_{n'}-j''}.\]
Here, $j'$ and $j''$ run over the following ranges.
When $n'\geq 2$, $(j',j'')$ run over all  $j'=0, \cdots , k'_1$, $j''=0, \cdots ,k'_n$.
When $n'=1$, $k'_1\geq 0$, $(j',j'')$ run over all $j'=0, \cdots , k'_1$, $j''=0, \cdots , k'_1$, $0\leq j'+j''\leq k'_1$.
When $n'=1$, $k'_1<0$, $(j',j'')$ run over all $j'=0, \cdots , k'_1$, $j''=0, \cdots , k'_1$, $0\geq j'+j''\geq k'_1$.

When $x_{i'_{n'}}\not =1$, set
\[ g_{j',j''}:=u_{i'_1}^{k'_1-j'}x_{i'_1}u_{i'_2}^{k'_2}\cdots u_{i'_{n'}}^{k'_{n'}}x_{i'_{n'}}\]
for $j'=0,\cdots , k'_1$, $j''=0$.
Here, there is an important notice: 

\bigskip

by condition (3), either $h_{j''',j}^*$ or $g_{j',j''}$ is not $1$.

\bigskip

Let
\[ 3^{-\frac{2M+r+s-j-j'}{2}}(\xi _m)_{M+r-j, M+s-j'}=\sum _{y}y,\]
\[ 3^{-\frac{2M+r+s-j''-j'''}{2}} (\xi _{m'})_{M+r'-j'',M+s-j'''}=\sum _{z}z\]
be decompositions, where $\{ y\}$, $\{ z\}$  are sets of mutually orthogonal non-zero scalar multiples of complete reduced words, respectively.

\bigskip

We also define a subset $V_1(m,r,s,j,y)$ of $\tilde{W}_M^0$ in the following way.

\textbf{Case 1.} When $w_{i_n}=1$, $V_1(m,r,s,j,y)$ is the set of all words $w$ satisfying the following conditions.

(1) The first letter of $w$ is neither the $(j-1)$-st letter of $h^*$ from the right nor the inverse of the $j$-th letter of $h$ from the right.

(2) The last letter of $w$ does not cancel with the first letter of $y$.

\textbf{Case 2.} When $w_{i_n}\not =1$,  $V_1(m,r,s,j,y)$ is the set of all words whose last letters do not cancel with the first letter of $y$.

\bigskip

We define a subset $V_2(m,r,s,j',y)$ of $\tilde{W}_M^0$ in the following way.

The set $V_2(m,r,s,j',y)$ consists of all words $w'$ satisfying the following conditions.

(1) The first letter of $w'$ does not cancel with the last letter of $y$.

(2)  The last letter of $w'$ is neither the $(j'-1)$-st letter of $g$ from the left nor the inverse of the $j'$-th letter of $g$ from the left.

\bigskip

We define a subset $V_3(m',r',s',j'',z)$  of $\tilde{W}_M^0$ in the following way.

\textbf{Case 1.} When $x_{i'_{n'}}=1$,  $V_3(m',r',s',j'',z)$ is the set of all words $w''$ satisfying the following conditions.

(1) The first letter of $w''$ is neither the $(j''-1)$-st letter of $g$ from the right nor the inverse of $j''$-th letter of $g$ from the right.

(2)  The last letter of $w''$ does not cancel with the first letter of $z$.

\textbf{Case 2.} When $x_{i'_{n'}}\not =1$, the set $V_3(m',r',s',j'',z)$ consists of all words whose last letters do not cancel with the first letter of $z$.

\bigskip

We define a subset  $V_4(m',r',s', j''',z)$ of $\tilde{W}_M^0$ in the following way.

The set $V_4(m',r',s', j''',z)$ consists  of all words $w'''$ satisfying the following conditions.

(1) The first letter of $w'''$ does not cancel with $z$.

(2) The last letter of $w'''$ is neither the $(j'''-1)$-st letter of $h^*$ from the left nor the inverse of the $j'''$-th letter of $h$ from the left.

\bigskip

Hereafter, if there is no danger of confusion, we sometimes abbreviate $V_1(m,r,s,j,y)$, $V_2(m,r,s,j',y)$, $V_3(m',r',s',j'',z)$ and $V_4(m',r',s',j''',z)$ to $V_1$, $V_2$, $V_3$ and $V_4$, respectively.
In this setting, we have the following claim.

\bigskip

\textbf{Claim 2.} We have
\begin{align*}
h^*\eta _1g\eta _2^*&= 3^{-2M}\sum _{j,j',j'',j'''}\sum _{r,s,m,r',s',m'}3^{\frac{j+j'+j''+j'''}{2}}u_{i_1}^{j'''}h_{j''',j}\sum _{y,z}\biggl( \sum _{w\in V_1(m,r,s,j,y)}w \biggr) \lambda _{m, 2M+r,2M+s}y \\
                                    &\biggl( \sum _{w'\in V_2(m,r,s,j',y)}w'\biggr) g_{j',j''}\biggl( \sum _{w''\in V_3(m',r',s',j'',z)}w'' \biggr) \mu _{m',2M+r',2M+s'}z \\
                                    &\biggl( \sum _{w'''\in V_4(m',r',s',j''',z)}w'''\biggr) u_{i_1}^{-j'''},
\end{align*}
where $j,j',j'',j'''$ run over the following ranges.
When $w_{i_n}=1$ and $n\geq 2$, $j,j'''$ run over all $j=0, \cdots , k_n$, $j'''=0, \cdots , k_1$.
When $w_{i_n}=1 $, $n=1$ and $k_1\geq 0$, $j,j'''$ run over all $j=0, \cdots , k_1$, $j'''=0, \cdots , k_1$, $0\leq j+j'''\leq k_1$.
When $w_{i_n}=1$, $n=1$ and $k_1<0$, $j,j'''$ run over all $j=0, \cdots , k_1$, $j'''=0, \cdots , k_1$, $0\geq j+j'''\geq k_1$.
When $w_{i_n} \not =1$, $j,j'''$ run over $j=0$, $j'''=0, \cdots , k_1$.

When $w_{i'_{n'}}=1$ and $n' \geq 2$, $j',j''$ run over all $j'=0, \cdots ,k_1$, $j''=0, \cdots , k'_n$.
When $w_{i'_{n'}}=1$, $n'=1$ and $k_1\geq 0$, $j',j''$ run over all $j'=0, \cdots , k'_1$, $j''0, \cdots , k'_1$, $j'+j''\leq k_1$.
When $w_{i'_{n'}}=1$, $n'=1$ and $k'_1<0$, $j',j''$ run over all $j'=0, \cdots , k'_1$, $j''=0, \cdots , k'_1$, $j'+j''\geq k'_1$.
When $w_{i'_{n'}}\not =1$, $j',j''$ run over all $j''=0$, $j'=0, \cdots , k'_1$.
 
\bigskip

\textit{Proof of Claim 2.} 
Obviously, we have
\begin{align*}
h^*\eta _1g\eta ^*_2 =3^{-\frac{8M+r+s+r'+s'}{2}}h^* &\sum _{m,r,s}\lambda _{m,2M+r,2M+s}(\xi _m) _{2M+r,2M+s} \\
                                      &g\sum _{m',r',s'}\mu _{m',2M+r',2M+s'}(\xi _{m'}) _{2M+r',2M+s'}.
\end{align*}
We would like to look at words appearing in $h^*(\xi _m) _{2M+r, 2M+s}g(\xi _{m'})_{2M+r',2M+s'}$ and to reduce each neighboring two words as possible.
Note that in this reduction, we do temporary think that the letters $w_i$ and $x_i$ are free from $u_i$.
Consider the component $h^*y'gz'$, where $y'$ is a linear component of $(\xi _m) _{2M+r,2M+s}$, $z'$ is a linear component of $(\xi _{m'})_{2M+r',2M+s'}$.
When we reduce each two neighboring two blocks in this word as possible, then the word becomes a linear sum of words of the form
\[ u_{i_1}^{j'''}h_{j''',j}wyw'g_{j',j''}w''zw''',\]
where $y$ is a component of $(\xi _m)_{M+r-j,M+s-j'}$, $z$ is a component of $(\xi _{m'})_{M+r'-j'',M+s'-j'''}$ and $w, w', w'',w'''$ are words of $\tilde{W}_M^0$ satisfying the following conditions.

\bigskip

(1) The word $w$ does not cancel with $h_{j''',j}^*$ or $y$.

(2) The word $w'$ does not cancel with $y$ or $g_{j',j''}$.

(3) The word $w''$ does not cancel with $g_{j',j''}$ or $z$.

(4) The word $w'''$ does not cancel with $z$ or $h_{j''',j}$.

\bigskip

We have to show that the above conditions are satisfied if and only if $(w,w',w'',w''') \in V_1\times V_2\times V_3\times V_4$.
For simplicity, we consider when $w_{i_n}=1$, $x_{i'_{n'}}=1$.
Other cases are shown in the same way (and the argument is much easier).
We show that if $w\not \in V_1$, then the word $w$ does not satisfy condition (1).
Since a word $w$ does not cancel with $h_{j''', j}$, the first letter of $w$ cannot be the inverse of the last letter of $h_{j''',j}$.
In order to cancel the $(j-1)$-st letter of $h$ from the right, there should be the inverse of that letter ahead of $w$.
Hence the first letter of $w$ cannot be the $(j-1)$-st letter of $h$.
Of course, if the last letter of $w$ cancel with $y$ , then the word $w $ cannot satisfy condition (1).
Hence any word $w\not \in V_1$ cannot contribute to the second summation.
On the other hand, if $w\in V_1$, then the word $w$ satisfies condition (1).
Similar statements hold for $w',w''$ and $w'''$.
Thus  we get the desired expression.
\qed

\bigskip

\textbf{Claim 3.}
Let $y$ be a linear component of $(\xi _m) _{M+r-j,M+s-j'}$, $z$ be a linear component of $(\xi _{m'})_{M+r'-j',M+s'-j''}$.
Consider a word $h^*_{j''',j}yxw'g_{j',j''}w''zw'''$.
If it satisfies $\tau (h^*_{j''',j}wyw'g_{j',j''}w''zw''')\not =0$,  then at least one of the following statements holds.

\bigskip

(1) We have $w=u_{i}^{\pm M}$.

(2) We have $w'=u_{i}^{\pm M}$.

(3) We have $w''=u_{i}^{\pm M}$.

(4) We have $w'''=u_{i}^{\pm M}$.

\bigskip

\textit{Proof of Claim 3.}
Assume that none of statements (1)--(4) holds.
We would like to show that $\tau (h_{j''',j}^*wyw'g_{j',j''}w''zw''')=0$.
Let
\[ h^*_{j''',j}=u_{i_1}^{k_1}w_{i_1}\cdots u_{i_n}^{k_n}w_{i_n},\]
\[ wyw'=u_{j_1}^{p_1}v_{j_1}^{q_1}\cdots u_{j_m}^{p_m}v_{j_m}^{q_m},\]
\[ g_{j',j''}=u_{i'_1}^{k'_1}x_{i'_1}\cdots u_{i'_{n'}}^{k'_{n'}}x_{i'_{n'}},\]
\[ w''zw'''=u_{j'_1}^{p'_1}v_{j'_1}^{q'_1}\cdots u_{j'_{m'}}^{p'_{m'}}v_{j'_{m'}}^{q'_{m'}}\]
be the completely reduced word expressions.

\textbf{Case 1}: neither $h^*_{j''',j}$ nor $g_{j',j''}$ is $1$.
In order to get the conclusion, it is enough to show that  the following.

\bigskip

(i) The vector $h_{j''',j}^*wyw'$ is a linear sum of words of the form $u_{i''_1}^{k''_1}v_{i''_1}^{l''_1}\cdots u_{i''_{n''}}^{k''_{n''}}v_{i''_{n''}}^{l''_{n''}}$.

(ii) The vector $g_{j',j''}w''zw'''$ is a linear sum of words of the form  $u_{i'''_1}^{k'''_1}v_{i'''_1}^{l'''_1}\cdots u_{i'''_{n'''}}^{k'''_{n'''}}v_{i'''_{n'''}}^{l'''_{n'''}}$.

(iii) Any pair of linear components of $h_{j''',j}^*wyw'$ and $g_{j',j''}w''zw'''$ does not cancel at all.

\bigskip

\textbf{Case 1-1}: both $w_{i_n}$ and $x_{i'_{n'}}$ are $1$.
In this case, we have statements (i)--(iii) without any assumption.

\bigskip

\textbf{Case 1-2}: neither  $w_{i_n}$ nor $x_{i'_{n'}}$  is $1$.
In Claim 2, we temporary thought that the letters $w_i$ and $x_i$ were free from $u_i$.
However, in this claim, we do not.
This may cause $h_{j''',j}^*$ and $w $ to cancel.
Nonetheless, the assumption that $w\not =u_i^{\pm M}$ ensures that if we reduce $h_{j''',j}^*wyw'$ as possible, it is either of the form
\[ u_{i_1}^{k_1}w_{i_1}\cdots u_{i_n}^{k_n}(w_{i_n}u_{i_1}^{p_1}w_{i_n}^*)w_{i_n}u_{j_2}^{p_2}v_{j_2}^{q_2}\cdots u_{j_m}^{p_m}v_{j_m}^{q_m}\]
($i_n \not =j_2$) or
\[ u_{i_1}^{k_1}w_{i_1}\cdots u_{i_n}^{k_n}w_{i_n}u_{j_1}^{p_1}v_{j_1}^{q_1}\cdots u_{j_m}^{p_m}v_{j_m}^{q_m}\]
($i_n\not =j_1$). Since $w_{i_n}$ is orthogonal to the subalgebra $\{ u_{i_n}\}''$ and normalizes it, by freeness, the vector $h_{j''',j}^*wyw'$ satisfies condition (i).
Similarly, since $w''\not =u_i^{\pm M}$, the vector $g_{j',j''}w''zw'''$ satisfies condition (ii).
Since $w'\not =u_i^{\pm M}$, if $k'_1=0$, then condition (iii) is shown by the same argument. 
If $k'_1\not =0$, then $w'$ and $g_{j',j''}$ does not cancel at all.
Hence condition (iii) is trivial.

\bigskip

\textbf{Case 1-3}: exactly onr of $w_{i_n}$ and $x_{i'_{n'}}$ is $1$.
This case is treated by the combination of the arguments in Cases 1-1 and 1-2.

\bigskip 

\textbf{Case 2}: exactly one of $h_{j''',j}^*$ and $g_{j',j''}$ is $1$.
For simplicity, we assume $g_{j',j''}=1$ (the other case is handled in the same way).
We have 
\[ \tau (h_{j''',j}^*wyw'g_{j',j''}w''zw''')=\tau (w'\cdot (w''zw'''h_{j''',j}^*wy)).\]
Since any component of $h_{j''',j}^*$ contains $v_j^{\pm 1}$, for $\tau (w'\cdot (w''zw'''h_{j''',j}^*wy))$, in order to be non-zero, at least one of $z$ and $y$ contains $v_j^{\pm 1}$.
Since neither $w'''$ nor $w$ is $u_i^{\pm M}$, by the same argument as that of Case 1,  any component of $w''zw'''h_{j''',j}^*wy$ contains at least one $v_j^{\pm 1}$ if we reduce it as possible.
Thus we have  $\tau (w'\cdot (w''zw'''h_{j''',j}^*wy))=0$.
\qed

\bigskip

Set
\[ \eta _{1,j,j'}:=3^{-M+\frac{j+j'}{2}}\sum _{r,s,m}\sum _y \Biggl( \sum _{w\in V_1}w\Biggr) \lambda _{m,2M+r,2M+s}y\Biggl( \sum _{w'\in V_2}w'\Biggr) .\] 
\[ \eta _{1,j,j'}':=3^{-M+\frac{j+j'}{2}}\sum _{r,s,m}\sum _y\Biggl( \sum _{i=1,2, \ s\in \{ \pm 1\} , \ u_i^{sM} \in V_1}  u_i^{sM}\Biggr) \lambda _{m,2M+r,2M+s}y \Biggl( \sum _{w'\in V_2}w'\Biggr) ,\]
\[ \eta _{1,j,j'}'':=3^{-M+\frac{j+j'}{2}}\sum _{r,s,m} \sum _y\Biggl( \sum _{w\in V_1, \ w\not =u_i^{sM}}w\Biggr) \lambda _{m,2M+r,2M+s}y\Biggl(  \sum _{i=1,2, \ s\in \{ \pm 1\} , \ u_i^{sM} \in V_2}  u_i^{sM}\Biggr) .\] 
\[ \eta _{2,j',j''}^*:=3^{-M+\frac{j''+j'''}{2}}\sum _{r',s',m'}\sum _z\Biggl( \sum _{w''\in V_3}w'' \Biggr) \mu _{m',2M+r',2M+s'}z\Biggl( \sum _{w'''\in V_4} w'''\Biggr),\]
\[ \eta _{2,j',j''}'^*:=3^{-M+\frac{j''+j'''}{2}}\sum _z\Biggl( \sum _{i=1,2, \ s\in \{ \pm 1\} , u_i^{sM}\in V_3}  u_i^{sM}\Biggr)  \mu _{m',2M+r',2M+s'} z \Biggl( \sum _{w'''\in V_4}w'''\Biggr) ,\]
\[ \eta _{2,j',j''}''^*:=3^{-M+\frac{j''+j'''}{2}}\sum _z\Biggl( \sum _{w''\in V_3, \ w''\not =u_i^{sM}}w''\Biggr) \mu _{m',2M+r',2M+s'}z\Biggl( \sum _{i=1,2, \ s\in \{ \pm 1\}, \ u_i^{sM}\in V_4}  u_i^{sM}\Biggr) .\]
Then we have
\begin{align*}
 \| h_{j''',j}^*\eta _{1,j,j'}' \| _2 &=\| \eta _{1,j,j'}'\| _2 \\
                                                        &\leq 4\cdot 3^{-M+\frac{j+j'}{2}} \| \sum _{r,s,m}\sum _y\lambda _{m,2M+r,2M+s}y\Biggl( \sum _{w'\in V_2}w' \Biggr) \| _2 \\
                                                        &\leq 4\cdot 3^{-M+\frac{j+j'}{2}}\| \sum _{r,s,m}\sum _y \lambda _{m,2M+r,2M+s}y\Biggl( \sum _{w'\in W_M^0, \ |yw'|=|y|+M}w'\Biggr) \| _2 \\
                                                        &=4\cdot 3^{-M+\frac{j+j'}{2}}\| \sum _{r,s,m}\lambda _{m,2M+r,2M+s}(\xi _m)_{M+r-j,2M+s-j'}\| _2 \\
                                                        &\leq  4\frac{C_0}{3^{M/2}}\Biggl( \sum _{r,s,m}|\lambda _{m, 2M+r,2M+s}|^2\Biggr) ^{1/2} \leq 4\frac{C_0^2\cdot 3^{k}}{3^{M/2}}\| \eta _1\| _2.
 \end{align*}
Similar statements holds for other three vectors.
We also have
\begin{align*}
\ &\| h_{j''',j}^*(\eta _{1,j,j'}-(\eta _{1,j,j'}'+\eta _{1,j,j'}'')) \| _2 \\
&\leq \| h_{j''',j}^* \eta _{1,j,j'} \| _2 \\
&\leq 3^{-M} \| \sum _{r,s,m}\sum _y\lambda _{m,2M+r,2M+s}\Biggl( \sum _{w\in V_1} w\Biggr) y \Biggl( \sum _{w'\in V_2} w'\Biggr) \| _2.
\end{align*}
By the same argument as above, the above left hand side is not greater than
\[ 3^{-M}\cdot 3^{M}\cdot C_0\Biggl( \sum _{r,s,m} |\lambda _{m,2M+r,2M+s}|^2\Biggr) ^{1/2} \leq C_0 ^2\| \eta _1\| _2.\]

Hence we have
\begin{align*}
| \langle \eta _1g, h \eta _2\rangle |&= | \tau (h^*\eta _1 g\eta _2^*)  | \\
                                                      &\leq \sum _{j,j',j'',j'''} \Biggl(| \tau (h_{j''',j}^*(\eta _{1,j,j'} '+\eta _{1,j,j'}'') g_{j',j''}\eta _{2,j,j'}^*)  \\
                                                      &+ \tau (h_{j''',j}^*(\eta _{1,j,j'}-(\eta _{1,j,j'}'+\eta _{1,j,j'}''))g_{j',j''}(\eta _{2,j',j''}'^*+\eta _{2,j',j''}''^*)| \\
                                                      &+| \tau (h_{j''',j}^*(\eta _{1,j,j'}-(\eta _{1,j,j'}'+\eta _{1,j,j'}''))g_{j',j''}(\eta _{2,j',j''}^*-(\eta _{2,j',j''}'^*+\eta _{2,j',j''}''^*)) | \Biggr) .
\end{align*}
By Claim 3, the third term in the above sum is zero.
Hence the right hand side is not greater than
\[ M^4\Biggl( 4\cdot 4\cdot \frac{C_0^4}{3^{M/2}}\| \eta _1\| _2 \| \eta _2\| _2 +0\Biggr) \leq  16 \cdot C_0^4\cdot \frac{M^4}{3^{M/2}}\| \eta _1\| _2 \| \eta _2\| _2.\]
Thus we are done.
\end{proof}

We also have the following.
\begin{lemm}
\label{6-1}
Let $g_1,g_2$ be words of $W_l^0$ for some non-negative integer $l\geq 0$.
Let $h$ be a vector satisfying the following condition.

\bigskip

The vector $h$ is of the form $u_{i'_1}^{k'_1}x_{i'_1}\cdots u_{i'_{n'}}^{k'_{n}}x_{i'_{n'}}$ for some $n'\geq 1$, $i'_1\not =\cdots \not =i'_{n'}$, $k'_1, \cdots , k'_n\in \mathbf{Z}$, where for any $t=1, \cdots , n'$, the operator $u_{i'_{t}}^{k'_{t}}x_{i'_{t}}$ satisfies either \textup{(c)} $k'_{t}\not =0$, $x_{i'_{t}}=1$ or \textup{(d)} $x_{i'_{t}}$ is a normalizing unitary of $\{ u_{i'_{t}}\}''$ in $R_{i'_t}$ which is orthogonal to $\{ u_{i'_{t}}\}''$.

\bigskip

Set $k:=|k'_1|+\cdots |k'_n|$ \textup{(}Actually, this depends on the presentation of $h$. However, we fix the presentation or take the minimum\textup{)}.
Then there exists a constant $C>0$, which does not depend either  on none of $x_{i_t}$'s, such that for any $M >2\mathrm{max}\{ l,k\}$, any vectors $\eta _1$, $\eta _2$ of  
\[ \mathrm{span}\{ (\xi _m)_{M+r, M+s}\mid  m\geq 0 , r,s\geq 0\}\vee \Biggl( \mathrm{span}\{ \xi ^{i,l,k}_{M+r, M+s}\mid i=1,2, \ l\in \mathbf{Z}\setminus \{ 0\}, \ k\in \mathbf{Z} , r,s\geq 0\} \cap L\Biggr) ,\] we have
\[ \langle \eta _1 (g_1-g_2) , h \eta _2 \rangle \leq CM^43^{-M/2}\| \eta _1\| _2\| \eta _2\| _2.\]
\end{lemm}
\begin{proof}
When $h\in W_l^0$, then the lemma is shown by the same argument as that of the proof of Lemmas 5.1, 5.2, 5.3 and 6.1 of Cameron--Fang--Ravichandran--White \cite{CFRW}.
When $h\not \in W_l^0$, then the lemma is shown by Lemma \ref{6-1}.
\end{proof}

\section{Asymptotic orthogonality property of the subalgebra}
In this section, by using the results of Sections 4 and 5, we show that the subalgebra has the asymptotic orthogonality property.
In order to achieve this, we have to reduce the general cases to the special case, that is, the case when the Haar unitaries come from generators of the irrational rotation $\mathrm{C}^*$-algebras.
\begin{lemm}
Let $\alpha $ be a free action of $\mathbf{Z}/m\mathbf{Z}$  on a diffuse separable abelian von Neumann algebra $C$ and $\beta $ be an ergodic action of $\mathbf{Z}$ on a diffuse separable abelian von Neumann algebra $D$.
Let $u,v$ be generating Haar unitaries of $C$ and $D$, respectively.
Then for any positive number $\epsilon >0$, there exists a normal injective *-homomorphism $\theta $ from $C\rtimes _\alpha \mathbf{Z}/m\mathbf{Z}$ into $D\rtimes _\beta \mathbf{Z}$ satisfying $\theta (C)=D$, $\| \theta (u) -v\| _2<\epsilon$.
\end{lemm}
\begin{proof}
Since both $u$ and $v$ are generating Haar unitary, the map $u^k \mapsto v^k$ extends to a *-isomorphism $\theta _0$  from $C$ onto $D$ satisfying $\theta _0 (u)=v$.
Since $\theta _0\circ \alpha \circ \theta _0^{-1}$ is a free action of $\mathbf{Z}/(m\mathbf{Z})$ on $D$, it is possible to find a partition $\{ p_i\}_{i=1}^m$ of unitary by projections in $D$ with
\[ \theta _0\circ \alpha \circ \theta _0^{-1}(p_i)=p_{i+1}\]
for any $i=1, \cdots ,m$, where $p_{m+1}=p_1$.
Then there exists a partition $\{ q_l\}_{l=1}^L$ of $p_1$ by projections in $D$  such that there exist complex numbers $\{ \lambda (i,l)\}_{i=1, \cdots , m, l=1, \cdots , L}\subset \mathbf{C}$ with
\[\| \sum _{i=1}^m\sum _{l=1}^L\lambda (i,l) \theta _0\circ \alpha ^{i-1}\circ  \theta _0^{-1}(q_l)-v\| _2<\frac{\epsilon }{2}.\]
Set 
\[ v_0:=\sum _{i=1}^m \sum _{l=1}^L \lambda (i,l) \theta _0 \circ \alpha ^{i-1}\circ \theta _0^{-1}(q_l).\]
Since the action $\beta $ of $\mathbf{Z}$ on $D$ is ergodic, there exists a unitary $w$ of $D\rtimes _\beta \mathbf{Z}$ satisfying the following two conditions.

\bigskip

(1) The unitary $w$ normalizes $D$.

(2) We have $w^{i-1}q_lw^{-i+1}=\theta _0\circ \alpha ^{i-1}\circ \theta _0^{-1}(q_l)$ for any $l=1, \cdots , L$, $i=1, \cdots , m+1$.

(3) We have $w^m=1$.

\bigskip

For $x\in C$, set
\[ \theta (x):=\sum _{i=1}^m  \theta _0\circ \alpha ^{i-1}\circ \theta _0^{-1}(p_1) w^{i-1}(\theta _0\circ \alpha ^{-i+1}(x))w^{-i+1}.\]
\textbf{Claim 1.} The map $\theta $ is a *-isomorphism from $C$ onto $D$.

\textit{Proof of Claim 1.} Notice that the map $\theta $ maps each $\alpha ^{i-1}\circ \theta _0^{-1}(q_l)$ to $\theta _0\circ \alpha ^{i-1}\circ \theta _0^{-1}(q_l)$ because we have
 \begin{align*}
 \ &\theta (\alpha ^{i'-1}\circ \theta _0^{-1}(q_l')) \\
    &=\sum _{i}\theta _0\circ \alpha ^{i-1}\circ \theta _0^{-1}(p_1)w^{i-1}(\theta _0\circ \alpha ^{-i+1}(\alpha ^{i'-1}\circ \theta _0^{-1}(q_{l'})))w^{-i+1} \\
    &=\sum _{i}\theta _0\circ \alpha ^{i-1}\circ \theta _0^{-1}(p_1) w^{i-1}(\theta _0\circ \alpha ^{-i+i'}\circ \theta _0^{-1}(q_{l'}))w^{-i+1} \\
    &=\sum _{i}\theta _0\circ \alpha ^{i-1}\circ \theta _0^{-1}(p_1) w^{i-1}(w^{i'-i}q_{l'}w^{-(i'-i)})w^{-i+1} \\
    &=\sum _{i}\theta _0\circ \alpha ^{i-1} \circ \theta _0^{-1}(p_1) \theta _0\circ \alpha ^{i'-1}\circ \theta _0^{-1} (q_{l'})  \\
    &=\theta _0\circ \alpha ^{i'-1}\circ \theta _0^{-1} (q_{l'}).
 \end{align*}
 Note that the third equality of the above computation follows from conditions (2) and (3).
On the other hand, the map $\mathrm{Ad}w^{i-1} \circ \theta _0\circ \alpha ^{-i+1}$ is a *-isomorphism from $C$ onto $D$. 
Hence so is its restriction to $C_{\alpha ^{i-1}\circ \theta _0^{-1}(q_l)}$.
Thus $\theta $ is a *-isomorphism from $C$ onto $D$.
\qed

\bigskip

\textbf{Claim 2.} We have $\theta \circ \alpha =(\mathrm{Ad}w|_D) \circ \theta $.

\textit{Proof of Claim 2.}
For $x\in C$, we have
\begin{align*}
\ &w\theta (x)w^{-1} \\
   &=w\Biggl( \sum _{i=1}^m \theta _0\circ \alpha ^{i-1}\circ \theta _0^{-1} (p_1) w^{i-1}\biggl( \theta _0\circ \alpha ^{-i+1}(x)\biggr) w^{-i+1}\Biggr) w^{-1} \\
   &=\sum _{i=1}^m (w^ip_1w^{-i}) (w^i \theta _0\circ \alpha ^{-i+1}(x)w^{-i}) \\
   &=\sum _{i=1}^m  \theta _0\circ \alpha ^i\circ \theta _0^{-1}(p_1) w^i(\theta _0\circ \alpha ^{-i} (\alpha (x)))w^{-i} \\
   &=\theta (\alpha (x)).
\end{align*}
\qed

\bigskip

\textbf{Claim 3.}
We have $\theta (\theta _0^{-1}(v_0))=v_0$.

\textit{Proof of Claim 3.}
As we have see in the proof of Claim 1, we have $\theta (\alpha ^{i-1}\circ \theta _0^{-1} (q_l)) =\theta _0\circ \alpha ^{i-1}\circ \theta _0^{-1}(q_l)$.
Since $v_0$ is a linear combination of $\theta _0\circ \alpha ^{i-1}\circ \theta _0^{-1}(q_l)$, we have $\theta (\theta _0^{-1}(v_0) )=v_0$.
\qed

\bigskip

By Claim 2, the *-isomorphism $\theta $ extends to a *-isomorphism form $C\rtimes _\alpha (\mathbf{Z}/(m\mathbf{Z})$ into $D \rtimes _\beta \mathbf{Z}$ satisfying $\theta (C)=D$, $\theta (\lambda _1^\alpha )=w$.
By Claim 3, we also have
\begin{align*}
\| \theta (u)-v)\| _2&\leq \| \theta (u-\theta _0^{-1}(v_0))\| +\| v_0-v\| \\
                                  &=\| u-\theta _0^{-1}(v_0) \| _2 +\| v_0-v\| _2 \\
                                  &=\| \theta _0 (u)-v_0\| _2+\| v_0-v\| _2 \\
                                  &=2\| v_0-v\| _2 \\
                                  &<\epsilon .
\end{align*}
\end{proof}
Let $N$ be the free product of two hyperfinite factor $R_i$ ($i=1,2$) of type $\mathrm{II}_1$ with respect to their traces.
For each $i=1,2$, choose a Haar unitary $w_i$ of $R_i$ which generates a Cartan subalgeba of $R_i$.
Set
\[ B:=\{ w_1+w_1^{-1}+w_2+w_2^{-1}\} '' \subset N.\]
For each non-negative integer $l\geq 0$, let $\chi _l^B$ be the sum of all reduced words of $\{ w_i^{\pm 1}\}_{i=1,2}$ with length $l$.
For each $i=1,2$, think of $R_i$ as an increasing union $\{  B_i\rtimes _{\alpha _i}G_k\} _{k=1}^\infty $ of von Neumann algebras  of type I, where $B_i:=\{ w_i\} ''$, $\{ G_k\}_{k=1}^\infty$ is an increasing sequence of finite abelian groups and $\alpha _i$ be a free ergodic probability measure preserving action of $\bigcup _kG_k$ on $B_i$.
Set 
\[ N_k:=(B_1\rtimes _{\alpha _1}G_k ) *(B_2\rtimes _{\alpha _2}G_k),\]
which is a von Neumann subalgebra of $N$.
\begin{prop}
\label{6-9}
Let $B\subset N$ be as above.
For $j=1, \cdots , J$, choose   $(x_n^j)\in (N^\omega \ominus B^\omega )\cap B'$.
Assume that for each $n$, there exists a positive integer $k_n>0$ with $x_n^j \in N_{k_n}$ for all $j=1, \cdots , J$.

Then there exists a family of weakly continuous  injective *-homomorphisms $\{ \theta _n :N_{k_n}\to M\} _{n=1}^\infty$  satisfying the following conditions.

\textup{(1)} For each $i=1,2$, we have $\theta _n(w_i)\to u_i$ as $n \to \omega$.

\textup{(2)} For any  $n$, any normalizing unitary $x\in B_i \rtimes _{\alpha _i} G_{k_n}$ of $B_i$ which is orthogonal to $B_i$,   the unitary $\theta _n(x)$ normalizes $\{ u_i\}''$ and is orthogonal to $\{ u_i\}''$.

\textup{(3)} We have $(\theta _n (x_n^j))\in (M^\omega \ominus A^\omega )\cap A'$.
\end{prop}
\begin{proof}
By the previous lemma, there exists a *-homomorphism $\overline{\theta} ^i_k$ from $B_i \rtimes _{\alpha _i} G_k$ into $\{ u_i,v_i\} ''$ satisfying $\overline{\theta }^i_k(B_i)=\{ u_i\}''$, $\| \overline{\theta} ^i_k(w_i)-u_i\|_2<2^{-k}$.
Set $\overline{\theta }_k:=\overline{\theta} ^1_k*\overline{\theta }^2_k$, which is a *-homomorphism from $N_k$ into $M$.

\bigskip

\textbf{Claim.} 
There exists an increasing sequence $\{ k'_n\}$ of natural numbers with the following conditions.

\bigskip

(1)  For any $n$, we have $k'_n>k_n$.

(2) For any $n$, we have  $\| E_A((\overline{\theta }_{k'_n} (x_n^j) )_n)\| _2 \leq \| x_n^j\| _2/2^n$.

(3) For any $n$, we have $\| \overline{\theta }_{k'_n}(w_i)-u_i\| _2<1/2^n$.

\bigskip

\textit{Proof of Claim.}
For each $n$ and $j$, since $x_n^j$ is orthogonal to $B$, $x_n^j$ can be approximated by a linear sum of vectors of the following forms.

\bigskip

(1) The vectors $b_1-b_2$, where $b_1$ and $b_2$ are words of $w_i^{\pm 1}$ of the same length.

(2) The words of $\lambda ^i_g$, $w_j^{\pm 1}$ for some $i,j=1,2$, $g\in (\bigcup _k G_k)\setminus \{0\}$ with at least one $\lambda ^i_g$.

\bigskip

If $g_t\not =0$ for some $t$, then we have
\[ E_A(\overline{\theta }_k(w_{i_1}^{l_1}\lambda _{g_1}^{i_1}\cdots w_{i_m}^{l_m}\lambda _{g_m}^{i_m}))=0,\]
which implies that the image of any vector of the second  form by $E_A\circ \overline{\theta }_k$ converges to $0$ as $k \to \infty$.
We also have $\overline{\theta }_k(w_{i_1}^{l_1}\cdots w_{i_m}^{k_m})\to u_{i_1}^{l_1}\cdots u_{i_m}^{l_m}$, which implies that any vector of the first form converges to a vector orthogonal to $A$.
Thus we have $E_A(\overline{\theta }_k (x_n^j) ) \to 0$ as $k \to \infty$.
We also have $\overline{\theta }_k(w_i)\to u_i$ as $k\to \infty$.
Thus Claim holds.
\qed

\bigskip

Set $\theta _n:=\overline{\theta }_{k'_n}$.

Then the family $\{ \theta _n:N_{k'_n}\to M\}$ is a family of weakly continuous injective *-homomorphisms satisfying conditions (1) and (2).
We show condition (3).
Since we have $\| E_A(\theta _n(x_n^j))\|  _2 <1/2^{n-1}\to 0$ as $n \to \omega$, we have $(\theta _n (x_n^j))\in M^\omega \ominus A^\omega $.
Next, we show that $\theta (x) $ commutes with $A$.
Take $y\in A$.
Since we have $\| \theta _n^{-1}(u_i)-w_i\| _2 =\| u_i -\theta _n(w_i)\| _2 \to 0$ as $n \to \omega$, the sequence $\{ \theta _n^{-1}(y)\}$ converges to an operator $z$ of $B$ in the strong operator topology.
Thus we have $x_n \theta _n^{-1}(y)-\theta _n^{-1}(y)x_n$ converges to $0$ as $n $ tends to $\omega$.
Hence $\theta (x)$ commutes with $A$.
\end{proof}

\begin{rema}
For each $x\in N$, set $\theta (x):=(\theta _n (E_{k_n}(x)))_n\in M^\omega$.
Then the map $\theta $ is a weakly continuous  injective *-homomorphism from $N $ into $M^\omega$.
\end{rema}

\begin{theo}
\label{6-10}
The subalgebra $B$ has the asymptotic orthogonality property.
\end{theo}
\begin{proof}
Choose $x^1=(x^1_n)$, $x^2=(x^2_n) \in (N^\omega \ominus B^\omega )\cap B'$ and $y^1,y^2\in N \ominus B$.
Let $\theta _n:N_{k_n}\to M$, $\theta :N \to M^\omega $ be injective *-homomorphisms chosen in the previous proposition corresponding to $(x_n^1)$, $(x_n^2)$.
Then we have $\theta (x^1), \theta (x^2)\in (M^\omega \ominus A^\omega )\cap A'$.
Hence by Lemmas \ref{6-1} and \ref{5-5}, we may assume that the vectors $\theta _n(x^1_n)$ and $\theta _n(x^2_n)$  lie in 
\begin{align*}
\ &\mathrm{span}\{ (\xi _m)_{2M_n+r, 2M_n+s}\mid  m\geq 0 , \ r,s\geq 0\} \\
   &\vee \Biggl( \mathrm{span}\{ \xi ^k_{2M_n+r, 2M_n+s}\mid k\in \mathbf{Z} , r,s\geq 0\} \cap L\Biggr) ,
\end{align*}
where $M_n \to \omega $ as $n \to \omega$.
On the other hand, when we regard $N$ as $\overline{\bigcup _k N_k}^{\mathrm{weak}}$, the Hilbert subspace $N \ominus B$ is linearly spanned by the following things.

\bigskip

(1) The vectors $b_1-b_2$, where $b_1$ and $b_2$ are words of $w_i^{\pm 1}$ of the same length.

(2) The words of $\lambda ^i_g$, $w_j^{\pm 1}$ for some $i,j=1,2$, $g\in (\bigcup _k G_k)\setminus \{0\}$ with at least one $\lambda ^i_g$.

\bigskip

Hence we may assume that $y_1$ and $y_2 $ are of the above forms.
Then by Proposition \ref{6-9} (1) (2), for each $n$, the vectors $\theta _n (y_1)$ and $\theta _n (y_2)$ satisfy assumptions of Lemmas \ref{6-1} and \ref{6-2}.
Hence by Lemmas \ref{6-1} and \ref{6-2}, we have $\tau (y_1^*(x^1_n)^*y_2x^2_n) \to 0$ as $n \to \omega$.
\end{proof}

\section{The subalgebra is maximal amenable.}
In this section, we show that the subalgebra $B$ is maximal amenable.
In order to achieve this, we use Proposition \ref{6-13}.
In the previous section, we have already shown that the subalgebra $B$ has the asymptotic orthogonality property.
Hence in order to show that the subalgebra is maximal amenable, it is enough to show that it is singular.
In order to achieve this, we show the mixing property.
\begin{defi}
\textup{(Definition 3.1 of Cameron--Fang--Mukherjee \cite{CFM})}
Let $A$ be a diffuse abelian von Neumann subalgebra of a factor $M$ of type $\mathrm{II}_1$.
The subalgebra $A$ is said to be mixing if for any $a, b\in M\ominus A$, any sequence $\{u _n\}$ of unitaries of $A$ which converges to $0$ weakly, we have $\| E_A(au_nb)\| _2\to 0$.
\end{defi}
It is known that if a subalgebra is mixing, then it is singular (Proposition 1.1 of Jolissaint--Stalder \cite{JS}).
\begin{lemm}
Let $w$ be a reduced word of $\{ w_1^{\pm 1}, w_2^{\pm 1}\}$ with $|w|=M>0$ and $g,h$ be words of the forms
\[ g=\lambda _{g_1}^{i_1}w_{i_1}^{k_1}\cdots \lambda _{g_n}^{i_n}w_{i_n}^{k_n},\]
\[ h=w_{j_1}^{l_1}\lambda _{h_1}^{j_1}\cdots w_{j_m}^{l_m}\lambda _{h_m}^{j_m},\]
respectively.
Suppose that at least one of $g_s$'s is non-zero and that one of $h_t$'s is non-zero.
Assume that they satisfy the following conditions.

\bigskip

\textup{(1)} The word $w $ is not $w_i^{\pm M}$.

\textup{(2)} The words $w_{i_n}^{k_n}$ and $w$ do not cancel, that is, we have $| w_{i_n}^{k_n}w|=|w_{i_n}^{k_n}||w|$.

\textup{(3)} The words $w$ and $w_{j_1}^{l_1}$ do not cancel. 

\bigskip

Then we have $E_B(gwh)=0$.

\end{lemm}
\begin{proof}
This is shown by the same argument as that of Claim 3 of the proof of Lemma \ref{6-2}.
\end{proof}
\begin{lemm}
Let $v=\sum _{p\geq M} \lambda _p w_p/\| w_p\| _2$ be a vector of $L^2B$ with $\| v\| _2=1$, where $w_p$ is the sum of all words of $w_i^{\pm 1}$ with length $p$.
Then there exists a positive constant $C_{g,h}>0$ with $\| E_B(gvh)\| _2\leq C_{g,h}3^{M/2}$.
\end{lemm}
\begin{proof}
If either $g$ or $h$ is in $W_l^0$, then there is nothing to show (Here, we use the mixing property of the radial MASA implicitly. See Sinclair--Smith \cite{SS} Theorems 3.1 and 5.1. Although the statements are slightly different, they essentially show the mixing property of the radial MASA).
Write $g$ and $h$ as the following way.
\[ g=\lambda _{g_1}^{i_1}w_{i_1}^{k_1}\cdots \lambda _{g_n}^{i_n}w_{i_n}^{k_n},\]
\[ h=w_{j_1}^{l_1}\lambda _{h_1}^{j_1}\cdots w_{j_m}^{l_m}\lambda _{h_m}^{j_m}.\]
Set $|g|:=|k_n|$, $|h|:=|l_1|$.
By the previous lemma, we have $\| E_B(gw_ph)\| _2 \leq 4 |g||h|$.
Thus we have
\begin{align*}
\| E_B(gvh) \| _2 &\leq 4|g||h| \biggl( \sum _{p\geq M}|\lambda _p|\frac{1}{3^{p/2}}\biggr) \\
                       &\leq 4|g||h|\biggl( \sum _{p\geq M}|\lambda _p|^2\biggr) ^{1/2}\biggl( \sum _{p\geq M}\frac{1}{3^p}\biggr) ^{1/2} \leq C_{g,h}\frac{1}{3^{M/2}}.
\end{align*}
\end{proof}
\begin{prop}
\label{7-1}
The subalgebra $B$ is mixing in $N$. 
In particular, it is singular in $N$.
\end{prop}
\begin{proof}
Notice that if $u^k \to 0$ weakly, then $u^k$ is approximated by operators $v $ of the form $\sum _{p\geq M_k}\lambda _p w_p/\| w_p\| _2$ in the strong operator topology, where $M_k \to \infty $.
Hence this is obvious by the previous lemma.
\end{proof}
\begin{theo}
The subalgebra $B$ is maximal amenable in $N$.
\end{theo}
\begin{proof}
By Theorem \ref{6-10}, the subalgebra $B$ has the asymptotic orthogonality property.
By Lemma \ref{7-1}, the subalgebra $B$ is a singular abelian von Neumann subalgebra of $N$.
Thus by Proposition \ref{6-13}, the subalgebra $B$ is maximal amenable in $N$.
\end{proof}


\begin{thebibliography}{99}
\bibitem{BC}
R. Boutonnet and A. Carderi, Maximal amenable von Neumann subalgebras arising from maximal amenable subgroups, Geom. Funct. Anal. \textbf{25} (2015) 1688--1705.
\bibitem{BH}
R. Boutonnet and C. Houdayer, Amenable absorption in amalgamated free product von Neumann algebras, preprint, arXiv:1606.00808.
\bibitem{CFM}
J. Cameron, J. Fang and K. Mukherjee, Mixing subalgebras of finite von Neumann algebras, New York J. Math. \textbf{19} (2013), 343--366.
\bibitem{CFRW}
J. Cameron, J. Fang, M. Ravichandran and S. White, The radial masa in a free group factor is maximal injective,  J. London Math. Soc. (2) \textbf{82} (2010), 787--809.
\bibitem{D}
K. Dykema, Free products of hyperfinite von Neumann algebras and free dimension, Duke Math. J. \textbf{69} (1993), 97--119. 
\bibitem{G}
L.Ge, On maximal injective subalgebras of factors, Adv. Math. \textbf{118} no. 1 (1996), 34--70.
\bibitem{GS}
A. Guionnet and D. Shlyakhtenko, Free monotone transport, Invent. Math. \textbf{197} no. 3 (2014), 613--661.
\bibitem{H}
C. Houdayer, Gamma stability in free product von Neumann algebras, Comm. Math. Phys. \textbf{336} (2015), 831--851.
\bibitem{HU}
C. Houdayer and Y. Ueda, Asymptotic structure of free product von Neumann algebras, preprint  arXiv:1503.02460, to appear in Math. Proc. Cambridge Philos. Soc.
\bibitem{JS}
P. Jolissaint, Y. Stalder, Strongly singular MASAs and mixing actions in finite von Neumann algebras,  Ergodic Theory Dynam. Systems  \textbf{28} no. 6  (2008), 1861--1878. 
\bibitem{O}
N. Ozawa, A remark on amenable von Neumann subalgebras in a tracial free product, Proc. Japan Acad. Ser. A Math. Sci. \textbf{91} no. 7 (2015), 104.
\bibitem{P}
S. Popa, Maximal injective subalgebras in factors associated with free groups, Adv. Math. \textbf{50} no. 1 (1983), 27--48. 
\bibitem{P2}
S. Popa, Orthogonal pairs of *subalgebras in finite von Neumann algebras, J. Oper. Theo. \textbf{9} no. 2 (1983),  253--268.
\bibitem{Rad}
R. Radulescu, Singularity of the radial subalgebra of $\mathcal{L}(F_N)$ and the Pukanszky invariant. Pacific. J. Math. \textbf{151} no. 2 (1991), 297--306.
\bibitem{SS}
A. Sinclair and R. Smith, The Laplacian MASA in a free group factor, Trans. Amer. Math. Soc. \textbf{355} no. 2  (2003), 465--475.
\bibitem{W}
C. Wen, Maximal amenability and disjointness for the radial masa, J. Funct. Anal. \textbf{270} no. 2 (2016), 787--801.
\end{thebibliography}
\end{document}